\documentclass[reqno,12pt]{amsart} 
\usepackage{esint,mathrsfs,amsmath,amssymb,amsfonts}

\usepackage[hyperfootnotes=false,hidelinks]{hyperref}

\usepackage{enumitem}

\usepackage[left=1.2in, right=1.2in, bottom=1.5in]{geometry}
 
\newtheorem{theorem}{Theorem}[section]
\newtheorem{lemma}[theorem]{Lemma}

\theoremstyle{definition}

\theoremstyle{remark}

\numberwithin{equation}{section}

\newcommand{\Norm}[1]{\lVert#1\rVert}

\newcommand{\N}{\mathcal{N}}

\newcommand{\R}{\mathbb{R}}

\newcommand{\dist}{\text{dist}}

\usepackage{cite}

\begin{document}
\title[Stability of Anomalous Dissipation]{Stability of Anomalous Dissipation \protect\\ for the Forced Three-Dimensional Navier--Stokes \protect\\ Equations under Geometric Perturbations}
\author{Changhong Li}
\address{Changhong Li: Academy of Mathematics and Systems Science, Chinese Academy of Sciences, Beijing, China.}
\email{lichanghong@amss.ac.cn}

\renewcommand{\abstractname}{Abstract}
\begin{abstract}
The energy dissipation in the inviscid limit is a central problem in turbulence theory. Kolmogorov's 1941 (K41) theory predicts nonvanishing energy dissipation as viscosity tends to zero, a phenomenon known as anomalous dissipation. Bru\`e and De Lellis constructed solutions exhibiting this phenomenon, but their construction relies on precise geometric conditions. Based on quasi-self-similar mixing, we prove that anomalous dissipation persists under small pure normal perturbations of the central curves satisfying area-preserving and equal-length conditions. The perturbations vanish near the spatial endpoints and at the initial and terminal times, and smallness is measured in a mixed norm with at most two time derivatives and total order at most six. We establish stability of the area-preserving maps and full local fields, including the mixed derivatives needed for H\"older estimates of the forcing, and obtain high-frequency energy concentration. A comparison estimate and a contradiction argument give a positive dissipation lower bound uniform over the admissible perturbations. Using the $(2+\frac12)$-dimensional framework, we construct solutions of the forced three-dimensional (3D) Navier--Stokes equations with anomalous dissipation for a sequence of viscosities tending to zero.
\end{abstract}

\maketitle

\section{Introduction}
\subsection{Background and motivation} 

\par
Turbulence is a highly complex and irregular flow state in fluid mechanics, with an energy cascade in which energy injected at large scales is transferred stepwise to smaller scales through $(u\cdot\nabla u)$ and eventually dissipated by viscosity at the Kolmogorov scale \cite{kolmogorov1941dissipation,kolmogorov1941local,kolmogorov1941degeneration,richardson1922weather}. Consider the three-dimensional (3D) incompressible Navier--Stokes equations on the torus $\mathbb{T}^3$ as the viscosity coefficient tends to zero:
\begin{equation}
    \begin{cases}
\partial_t u^\mu + u^\mu \cdot \nabla u^\mu + \nabla p^\mu = \mu \Delta u^\mu + f^\mu, \\
\operatorname{div} u^\mu = 0,\\
u^\mu(\cdot,0) = u_0^\mu,
\end{cases}\qquad (x,t)\in\mathbb{T}^3\times[0,T],
\label{eq:NS}
\end{equation}%
where $u^\mu$ is the velocity, $p^\mu$ the pressure, $f^\mu$ the forcing, and $u_0^\mu$ the initial datum. 
For smooth solutions, the energy equality reads
\[
\frac{\rm d}{\rm{d} t}\frac{1}{2}\|u^\mu\|_{L^2}^2 = -\mu\|\nabla u^\mu\|_{L^2}^2 + \int_{\mathbb{T}^3} f^\mu \cdot u^\mu \, \mathrm{d}x.
\]
From this equality, the viscous energy dissipation rate per unit mass is determined by $\mu\|\nabla u^\mu\|_{L^2}^2$. 
Kolmogorov's 1941 (K41) theory \cite{kolmogorov1941dissipation,kolmogorov1941local,kolmogorov1941degeneration} predicted nonvanishing viscous dissipation at high Reynolds numbers as viscosity tends to zero. For a sequence $\mu_m\to0$, we use the following time-integrated formulation:
\begin{equation}
    \liminf_{m\to\infty}\mu_m\int_0^T\|\nabla u^{\mu_m}\|_{L^2}^2\,dt>0.
    \label{eq:AD}
\end{equation}
We call this time-integrated property \emph{anomalous dissipation} for the chosen viscosity sequence.
\par
For classical Navier--Stokes solutions, if the initial data converge strongly in $L^2$, the forces converge strongly in $L^1_tL^2_x$, and the corresponding limiting Euler problem has a solution in $C([0,T];C^1)$, the relative energy estimate rules out anomalous dissipation on $[0,T]$ \cite[Appendix A]{brue2023anomalous}. Experiments and numerical simulations provide evidence for anomalous dissipation \cite{sreenivasan1998update,kaneda2003energy}, but constructing such Navier--Stokes solutions remains a major challenge. Our approach uses passive scalar mixing and a suitable viscosity sequence. Consider the transport equation
\begin{equation}
    \partial_t\rho+v\cdot\nabla\rho=0,
    \label{eq:T}
\end{equation}
where $v$ is a two-dimensional (2D) divergence-free velocity field and $\rho$ a passively transported scalar.
\par
Alberti, Crippa and Mazzucato \cite{alberti2019exponential} constructed a quasi-self-similar mixing family $\{(\rho_n,v_n)\}_{n\in\mathbb{N}}$ achieving exponential mixing \cite{alberti2014exponential,alberti2019loss}. Drivas et al. \cite{drivas2022anomalous} proved anomalous dissipation for passive scalars. Jeong and Yoneda \cite{jeong2021vortex,jeong2022quasi} studied related enhanced-dissipation mechanisms in the $(2+\frac12)$-dimensional framework. This framework reduces 3D Navier--Stokes to a 2D velocity equation and a passive-scalar equation driven by that velocity:
\begin{equation}
    \begin{cases}
        \partial_{t}v^{\mu}+v^{\mu}\cdot\nabla v^{\mu}+\nabla q^{\mu}=\mu\Delta v^{\mu}+g^\mu,\\
        \operatorname{div}v^{\mu}=0,\\
        \partial_{t}\theta^{\mu}+v^{\mu}\cdot\nabla \theta^{\mu}=\mu\Delta \theta^{\mu},     
    \end{cases}
    \label{nsmu}
\end{equation}
where $v^\mu$ and $\theta^\mu$ depend only on $(x_1,x_2,t)$.
\par
In 2023, Bru\`e and De Lellis \cite{brue2023anomalous} embedded the quasi-self-similar structure of Alberti et al. \cite{alberti2019exponential} into this framework and constructed a family of solutions to the forced 3D Navier--Stokes equations that exhibit anomalous dissipation. However, the construction relies on precise geometric conditions \cite{alberti2019exponential}. This raises the question of structural stability under small perturbations. The present stability analysis shows that anomalous dissipation persists under the small pure normal perturbations considered below. The proof uses the area-preserving and equal-length conditions and the zero-mean property of the scalar. Other constructions of anomalous dissipation include forced Navier--Stokes solutions with uniform velocity bounds near the Onsager exponent \cite{brue2022onsager}. For passive scalars, stochastic Lagrangian analysis \cite{colombo2022anomalous} and fractal homogenisation \cite{armstrong2023anomalous} provide related constructions.

\subsection{Assumptions}
Let the original central curves $\{\Gamma_i\}$ be parametrised by $\gamma_i(t,s)$ with $|\partial_s\gamma_i|=l(t)$ and arclength element $d\sigma=l(t)ds$. Define the unit tangent $\tau_i=\partial_s\gamma_i/|\partial_s\gamma_i|$, the unit normal $\eta_i=\tau_i^\perp$, and the curvature $\kappa_i$ by $\partial_s\tau_i=-l\kappa_i\eta_i$. Here $J(a,b)=(-b,a)$ and $\tau_i^\perp=J\tau_i$. In particular, we have
\begin{equation}\label{eq:frenet}
\partial_s\tau_i=-l\kappa_i\eta_i,\qquad
\partial_s\eta_i=l\kappa_i\tau_i.
\end{equation}
We consider pure normal perturbations of the form
\begin{equation}\label{eq:normalgraph}
    \tilde{\gamma}_i(t,s)=\gamma_i(t,s)+h_i(t,s)\eta_i(t,s).
\end{equation}%
We take $L=1$ by normalising the interior parameter interval to $[0,1]$. Denote by $\tilde{\gamma}_i^{\mathrm{arc}}$ the constant-speed reparametrisation of the normal graph $\tilde\gamma_i$, with the same image $\tilde\Gamma_i$.

Let the two generating curves and the six rotated configurations ($N=6$) be those of \cite[Sections 8.2 and 8.10--8.12]{alberti2019exponential}, translated to $[0,1]^2$. In this construction, the two generators have endpoints
\[
\gamma_1(t,0)=\gamma_2(t,0)=(1/2,0),\qquad
\gamma_1(t,1)=(1/2,1),\quad\gamma_2(t,1)=(1,1/2).
\]
These curves are smooth, proper and embedded, with the same speed $l(t)>0$, and divide the square into regions of time-independent areas. In a fixed neighbourhood of the boundary and outside the square, their images are fixed straight rays normal to the sides. At $t=1$, the part of each curve inside the square consists of twenty-five translated and rotated copies of the initial arcs, scaled by $1/5$. The remaining four configurations and their perturbations are obtained by the prescribed rotations of the two generators about $(1/2,1/2)$.

\subsection{Main results}

Let $\{(\rho_n,v_n)\}_{n\in\mathbb N}$ be the quasi-self-similar family of Alberti, Crippa and Mazzucato \cite{alberti2019exponential} with central curves $\Gamma_i(t)$ ($i=1,\dots,6$) satisfying $|\partial_s\gamma_i|=l(t)$. Apply the pure normal perturbation \eqref{eq:normalgraph}, where the perturbations $h_i$ of the two generating curves satisfy the following conditions:
\begin{enumerate}[label=(\roman*)]
\item Regularity and smallness: $h_i\in C^\infty([0,1]^2)$, and
\begin{equation}\label{eq:smallness}
\epsilon=\max_{i=1,2}\sum_{a=0}^2\sum_{b=0}^{6-a}
 \|\partial_t^a\partial_s^b h_i\|_{L^\infty([0,1]^2)}\le\epsilon_0,
\end{equation}
which controls, for $i=1,2$,
\[
\sup_{t\in[0,1]}\left(
\|h_i(t,\cdot)\|_{C_s^6}
+\|\partial_t h_i(t,\cdot)\|_{C_s^5}
+\|\partial_t^2 h_i(t,\cdot)\|_{C_s^4}\right).
\]
The perturbations are smooth, and the uniform estimates use only the derivatives in \eqref{eq:smallness}. In particular, smallness in $C^2([0,1];C^6([0,1]))$ is sufficient.
\item Boundary conditions: $\operatorname{supp}h_i(t,\cdot)\subset K_i^{\mathrm{par}}\subset(0,1)$ for all $t\in[0,1]$, where $K_i^{\mathrm{par}}$ is a fixed compact set. We extend $h_i$ by zero to all $s\in\R$.
\item Initial and terminal conditions: $h_i(0,s)=h_i(1,s)=0$ for all $s\in[0,1]$. Their time derivatives need not vanish at these endpoints.
\item Area-preserving condition: for all $t\in[0,1]$,
\begin{equation}\label{eq:area-constraint}
\mathcal B_i(h;t):=\int_{\Gamma_i(t)}h_i\,d\sigma
 +\frac12\int_{\Gamma_i(t)}h_i^2\kappa_i\,d\sigma=0.
\end{equation}
\item Equal-length condition: $\tilde L_i(t)=\tilde L_{i'}(t)$ for all $t\in[0,1]$ and $i,i'=1,\dots,6$, where $\tilde L_i(t)$ is defined by
\begin{equation}\label{6}
\tilde L_i(t)=\int_0^1\sqrt{l(t)^2(1+\kappa_i(t,s)h_i(t,s))^2
 +(\partial_s h_i(t,s))^2}\,ds.
\end{equation}
We write $\tilde l(t)=\tilde L_1(t)=\tilde L_2(t)$.
\end{enumerate}
Conditions (i)--(v) are sufficient for the construction. Denote this family by $\mathcal H_{\epsilon_0}$. The fixed data are the reference geometry, patching rules, supports, tube width, scalar template, cutoff and time change. Unless otherwise indicated, constants $C$ depend only on these data and are independent of the admissible perturbation and the generation index. Untilded fields correspond to $h=0$, with the same cutoffs and time change. We use $\mathbb N=\{0,1,2,\dots\}$ and take $m\ge1$ in the viscosity sequence.

\noindent
\begin{lemma}\label{1.1}
\textbf{(Existence and stability of the perturbed basic family)}
\par
Under conditions (i)--(v), there exist constants $\epsilon_0>0$ and $C>0$, depending only on the fixed data, such that for $\epsilon\le\epsilon_0$, the perturbed local velocity field and scalar field $(\tilde V_i,\tilde\Theta_i)$ constructed from the perturbed curves $\tilde\Gamma_i$ satisfy on $[0,1]^2\times[0,1]$ the transport equation
\begin{equation}
    \partial_t\tilde\Theta_i+\tilde V_i\cdot\nabla\tilde\Theta_i=0,
    \label{1.2-1}
\end{equation}
and have the following properties:
\par
(1) Regularity of the velocity field: $\tilde V_i\in C^\infty([0,1];C^3([0,1]^2;\mathbb{R}^2))$, $\operatorname{div}\tilde V_i=0$, and $\tilde V_i$ is tangent to the boundary $\partial[0,1]^2$.
\par
(2) Zero mean and unit norm of the scalar field: $\tilde\Theta_i\in C^\infty([0,1];C^3([0,1]^2))$ satisfies
\begin{align}
    \int_{(0,1)^2}\tilde\Theta_i(x,t)\,\mathrm dx&=0,\\
    \int_{(0,1)^2}\tilde\Theta_i(x,t)^2\,\mathrm dx&=1.
    \label{1.2-2}
\end{align}
\par
(3) Stability, with $C_t^aC_x^b=C^a([0,1];C^b([0,1]^2))$:
\begin{align}
\|\tilde V_i-V_i\|_{C_tC_x^3}
+\|\partial_t(\tilde V_i-V_i)\|_{C_tC_x^2}&\le C\epsilon,\notag\\
\|\tilde\Theta_i-\Theta_i\|_{C_tC_x^3}
+\|\partial_t(\tilde\Theta_i-\Theta_i)\|_{C_tC_x^2}&\le C\epsilon.
\label{1.2-3}
\end{align}
\par
(4) Quasi-self-similar patching condition: for every $Q\in\mathcal Q(5)$ there exists $j=j(Q,i)\in\{1,\dots,N\}$ such that
\begin{equation}
    \tilde\Theta_i(x,1)=\tilde\Theta_j\bigl(5(x-r_Q),0\bigr),\quad x\in Q,
    \label{1.2-4}
\end{equation}
where $r_Q$ is the lower-left corner and $j$ includes the prescribed rotation. The fields agree on open neighbourhoods of all spatial interfaces used in the patching.
\end{lemma}
\par

\noindent
\begin{theorem}\label{1.2}
\textbf{(Estimates of the perturbed quasi-self-similar family)}
\par
Under the assumptions of Lemma \ref{1.1}, using the quasi-self-similar construction and the Peano curve patching from \cite{alberti2019exponential}, we construct global velocity fields $\tilde v_n$ and scalar fields $\tilde\rho_n$ ($n\in\mathbb{N}$) defined on $[0,1]^2\times[0,1]$ such that:
\par
(1) Velocity estimates: $\tilde v_n\in C^\infty([0,1];C^3([0,1]^2;\mathbb{R}^2))$, $\operatorname{div}\tilde v_n=0$, and for every $\alpha\in(0,1)$ and $k=0,1$,
\begin{equation}
    \|\partial_t^k\tilde v_n\|_{L^\infty([0,1];\,C^\alpha(\mathbb{T}^2))}\le C(\alpha,k)5^{(\alpha-1)n}.
    \label{1.3-1}
\end{equation}
Moreover $\|\Delta\tilde v_n\|_{C_tC_x^\alpha}\le C(\alpha)5^{(1+\alpha)n}$. The corresponding differences from $v_n$ satisfy the same estimates with an additional factor $\epsilon$.
\par
(2) Scalar field estimates: $\tilde\rho_n\in C^\infty([0,1];C^3([0,1]^2))$, and for every $t\in[0,1]$, $n\in\mathbb{N}$,
\begin{align}
    \int_{(0,1)^2}\tilde\rho_n(x,t)\,\mathrm dx&=0,\\
    \int_{(0,1)^2}\tilde\rho_n(x,t)^2\,\mathrm dx&=1,
    \label{1.3-2}
\end{align}
and there exist constants $C_0,C_1>0$, depending only on the fixed data, such that for all $\epsilon\le\epsilon_0$,
\begin{align}
   \|\tilde\rho_n\|_{L^\infty} &\le C_0+C_1\epsilon,\\
   \|\nabla\tilde\rho_n\|_{L^\infty} &\le (C_0+C_1\epsilon)5^n.
   \label{1.3-3}
\end{align}
The scalar has zero mean on each tile and, for some constants $c_G,C_G,C_H>0$ depending only on the fixed data, satisfies
\begin{equation}\label{eq:scalar-mixing}
c_G5^n\le\|\nabla\tilde\rho_n\|_{L^2}\le C_G5^n,
\qquad \|\tilde\rho_n\|_{\dot H^{-1}}\le C_H5^{-n}.
\end{equation}
\par
(3) Compact support: for every fixed $n\in\mathbb N$ there exists a compact set $K_n\subset(0,1)^2$ such that for all $t\in[0,1]$,
\begin{equation}
    \operatorname{supp}\tilde v_n(\cdot,t)\cup\operatorname{supp}\tilde\rho_n(\cdot,t)\subset K_n.
    \label{1.3-4}
\end{equation}
\par
(4) Recurrence relation:
\begin{equation}
    \tilde\rho_n(x,1) = \tilde\rho_{n+1}(x,0),\quad \forall n\in\mathbb{N}.
    \label{1.3-5}
\end{equation}
\end{theorem}

\par
From the above quasi-self-similar construction we obtain a frequency concentration estimate for the global scalar field and prove the uniform persistence of the passive-scalar dissipation lower bound, i.e., Theorem \ref{1.3}.
\par
\noindent
\begin{theorem}\label{1.3}
\textbf{(Stability of the anomalous dissipation lower bound)}
\par
Under conditions (i)--(v), let the time sequence be $t_n=1-(n+1)^{-2}$, with $\Delta t_n=t_{n+1}-t_n$, and let $\eta\in C^\infty([0,1))\cap C([0,1])$ be the non-decreasing time reparametrisation from Lemma~\ref{4.1}, with $\eta(t_n)=t_n$ and all derivatives of positive order vanishing at each $t_n$.
Define the smoothed velocity field by
\begin{equation}
 \tilde v^m(x,t)=
 \begin{cases}
 \displaystyle\frac{\eta'(t)}{\Delta t_n}
 \tilde v_n\!\left(x,\frac{\eta(t)-t_n}{\Delta t_n}\right),
 &t\in[t_n,t_{n+1}),\quad 0\le n\le m,\\[1ex]
 0,&t\in[t_{m+1},1].
 \end{cases}
    \label{1.4-1}
\end{equation}
and choose the viscosity coefficient $\mu_m = m^{10}5^{-2m}$. Let $\tilde\theta^m$ be the solution of the transport equation with diffusion
\begin{equation}
    \partial_t\tilde\theta^m + \tilde v^m\cdot\nabla\tilde\theta^m = \mu_m\Delta\tilde\theta^m,\quad \tilde\theta^m(x,0) = \tilde\rho(x,0),
    \label{1.4-2}
\end{equation}
where $\tilde\rho$ is defined for $t<1$ by the same smoothing procedure applied to $\{\tilde\rho_n\}$. Its initial value $\rho_0(\cdot,0)$ is independent of the perturbation. Then there exists $C_*>0$, depending only on the fixed data, such that for every $h\in\mathcal H_{\epsilon_0}$,
\begin{equation}
    \liminf_{m\to\infty}\mu_m \int_0^1 \|\nabla\tilde\theta^m(\cdot,t)\|_{L^2(\mathbb{T}^2)}^2 dt \ge C_*.
    \label{1.4-3}
\end{equation}
\end{theorem}

\par
We now present the main theorem of this paper.
\par
\noindent
\begin{theorem}\label{1.4}
\textbf{(Uniform persistence of anomalous dissipation for the forced 3D Navier--Stokes equations)}
\par
Under conditions (i)--(v), in the $(2+\frac12)$-dimensional flow framework, define the three-dimensional velocity field and forcing by
\begin{align}
     \tilde u^{\mu_m}(x,t) &= \bigl(\tilde v^m(x_1,x_2,t),\; \tilde\theta^m(x_1,x_2,t)\bigr),\\
     \tilde f^{\mu_m}(x,t) &= \bigl(\tilde g^m(x_1,x_2,t),\; 0\bigr),
     \label{1.5-1}
\end{align}
where $\tilde u^{\mu_m}$ solves
\begin{equation}
    \begin{cases}
\partial_t \tilde u^{\mu_m} + \tilde u^{\mu_m} \cdot \nabla \tilde u^{\mu_m} + \nabla \tilde p^{\mu_m} = \mu_m \Delta \tilde u^{\mu_m} + \tilde f^{\mu_m}, \\
\operatorname{div} \tilde u^{\mu_m} = 0,\\
\tilde u^{\mu_m}(\cdot,0) = u_0,
\end{cases}\qquad (x,t)\in \mathbb{T}^3\times[0,1],
\label{1.5-2}
\end{equation}
with $\tilde p^{\mu_m}=0$ and $\tilde g^m=\partial_t\tilde v^m + \tilde v^m\cdot\nabla\tilde v^m - \mu_m\Delta\tilde v^m$. The initial datum $\tilde u^{\mu_m}(\cdot,0) = (v_0,\theta_0)\in C^\infty(\mathbb{T}^3)$ satisfies $v_0=0$ and $\theta_0=\rho_0(\cdot,0)$ and is independent of $m$ and the perturbation. Then:
\par
(1) Regularity: $\tilde u^{\mu_m},\tilde f^{\mu_m}\in C^\infty(\mathbb{T}^3\times[0,1])$ for each fixed $m$ and $h$, and for every $\alpha\in(0,1)$ there exists a constant $\tilde C(\alpha)$ independent of $m$ and the admissible perturbation such that
\begin{equation}
    \|\tilde f^{\mu_m}\|_{C([0,1];C^\alpha(\mathbb{T}^3))} \le \tilde C(\alpha).
    \label{1.5-3}
\end{equation}
For each admissible perturbation, the forces converge in this space to $\tilde f$, with convergence uniform over the admissible family. In addition,
\[
\sup_{m\ge1}\|\tilde f^{\mu_m}-f^{\mu_m}\|_{C_tC_x^\alpha}
 +\|\tilde f-f\|_{C_tC_x^\alpha}\le C(\alpha)\epsilon.
\]
\par
(2) Stability: there exist $\epsilon_0>0$ and $C_*>0$, depending only on the fixed data, such that for every $h\in\mathcal H_{\epsilon_0}$,
\begin{equation}
    \liminf_{m\to\infty}\mu_m \int_0^1 \|\nabla\tilde u^{\mu_m}\|_{L^2(\mathbb{T}^3)}^2 dt \ge C_*.
    \label{1.5-4}
\end{equation}
\end{theorem}

\par
Theorem \ref{1.4} shows that the Bru\`e--De Lellis construction retains anomalous dissipation under the pure normal perturbations in $\mathcal H_{\epsilon_0}$. The lower bounds in Theorems~\ref{1.3} and \ref{1.4} also hold for all sufficiently large $m$, with a threshold independent of the perturbation. Here stability refers to this uniform persistence together with the stated forcing difference estimate.

\textbf{Organization.} The paper is organised as follows. Section~2 introduces pure normal perturbations of the geometric modules in \cite{alberti2019exponential} and proves the stability of the area-preserving maps and the full local fields in finite mixed norms, with uniform constants. Section~3 verifies agreement at interfaces and patches the perturbed blocks through the Peano curve construction to form a global quasi-self-similar family, establishing H\"older estimates for the velocity and gradient growth for the scalar. Following the time partition and forcing construction of \cite{brue2023anomalous}, Section~4 smooths the construction in time and, by choosing $\mu_m=m^{10}5^{-2m}$, proves the uniform boundedness and convergence of the forcing in $C^\alpha$. Section~5 combines frequency concentration with a comparison argument up to a stopping time to prove the uniform anomalous dissipation lower bound. Finally, Section~6 embeds this two-dimensional result into the $(2+\frac12)$-dimensional framework of \cite{brue2023anomalous} to prove Theorem~\ref{1.4}.
\par
\section{Perturbed geometric structure}\label{sectionPerturbed geometric structure}
\par
Based on the geometric framework and the perturbation assumptions from Section 1, this section constructs for the perturbed central curves $\tilde{\Gamma}_i$ compatible area-preserving maps $\tilde{\Phi}_i$ and establishes their stability in the mixed norms specified below. From these estimates we deduce spatial $C^3$ stability of the local velocity and scalar fields $(\tilde{V}_i,\tilde{\Theta}_i)$ and spatial $C^2$ stability of their first time derivatives, thereby providing the foundation for the global quasi-self-similar construction in the next section.
\par
Fix a bounded open neighbourhood $W$ of $[0,1]^2$. We compare supports only in $\overline W$. Coordinate maps are estimated on compact parameter regions, and their inverses on the common physical regions specified below. The fixed straight rays allow us to choose the parameter intervals in $s$ uniformly. For $0\le p\le k$, define
\begin{equation}\label{eq:triangular-norm}
\N_{p,k}(F)=\sum_{a=0}^{p}\sum_{a+|\beta|\le k}
\Norm{\partial_t^aD_z^\beta F}_\infty.
\end{equation}
Here $z=s$ for curves, $z=(s,y)$ for coordinate maps, and $z=x$ for inverse maps. The norm uses at most $p$ time derivatives and total order at most $k$. In particular, $\N_{2,6}(h_i)\le\epsilon$.

\begin{lemma}\label{2.1}
There exist constants $\epsilon_0>0$, $C>0$ and $r>0$, depending only on the fixed data, such that under conditions (i)--(v) with $\epsilon\le\epsilon_0$, the perturbed central curves have parametrisations $\tilde{\gamma}_i^{\mathrm{arc}}$ with common speed $\tilde{l}(t)=\tilde L_1(t)=\tilde L_2(t)$. For each $t\in[0,1]$, the map defined by
\begin{align}
\tilde{\Phi}_i(t,s,y)&=\tilde{\gamma}_i^{\mathrm{arc}}(t,s)+\tilde{\beta}_i(t,s,y)\tilde{\eta}_i(t,s),
\label{2-8}\\
\tilde{\beta}_i(t,s,y)&=
\frac{2y/\tilde{l}(t)}{1+\sqrt{1+2\tilde{\kappa}_i(t,s)y/\tilde{l}(t)}}
\label{eq:beta}
\end{align}
is a smooth area-preserving diffeomorphism from $D=\R\times(-r,r)$ onto its image. The radius $r$ is independent of $t$ and the admissible perturbation.
Here $\tilde\eta_i$ and $\tilde\kappa_i$ are the unit normal and curvature of
$\tilde\gamma_i^{\mathrm{arc}}$, with the convention in \eqref{eq:frenet}.
Denote its inverse by $\tilde\Psi_i(t,\cdot)$.
On the compact parameter regions and common physical neighbourhoods constructed below, the map and its inverse satisfy
\begin{align}
\N_{2,4}(\tilde{\Phi}_i-\Phi_i)+\N_{2,4}(\tilde{\Psi}_i-\Psi_i)
 &\le C\epsilon,\label{2-9}\\
\N_{2,4}(\tilde{\Phi}_i)+\N_{2,4}(\tilde{\Psi}_i)&\le C.
\label{eq:map-uniform}
\end{align}
\end{lemma}

\begin{lemma}\label{2.2}
Under the assumptions of Lemma \ref{2.1}, let $g\in C_c^\infty((-r,r))$ be an even cutoff such that $g=1$ near $[-r/2,r/2]$ and $g=0$ for $|y|\ge3r/4$.
Fix a nonzero even $\bar{\Theta}\in C_c^\infty((-1/2,1/2))$ with $\int \bar{\Theta}=0$ and $\int\bar{\Theta}^2=1$, and set
$c_\Theta=(r\int \bar{\Theta}^2)^{-1/2}$. The construction in Section~\ref{sec:complete-fields} gives smooth fields $(\tilde{V}_i,\tilde{\Theta}_i)$ on a neighbourhood of $[0,1]^2$ satisfying
\[
\operatorname{div} \tilde{V}_i=0,\qquad \partial_t\tilde{\Theta}_i+\tilde{V}_i\cdot\nabla\tilde{\Theta}_i=0,
\qquad \tilde{\Theta}_i(\tilde{\Phi}_i(t,s,y),t)=c_\Theta \bar{\Theta}(y/r).
\]
With norms taken over $[0,1]$ in time and $[0,1]^2$ in space, we have
\begin{align}
\|\tilde V_i-V_i\|_{C_tC_x^3}
+\|\partial_t(\tilde V_i-V_i)\|_{C_tC_x^2}&\le C\epsilon,\notag\\
\|\tilde\Theta_i-\Theta_i\|_{C_tC_x^3}
+\|\partial_t(\tilde\Theta_i-\Theta_i)\|_{C_tC_x^2}&\le C\epsilon,\label{2-11}\\
\|\tilde V_i\|_{C_tC_x^3}
+\|\partial_t\tilde V_i\|_{C_tC_x^2}&\le C,\notag\\
\|\tilde\Theta_i\|_{C_tC_x^3}
+\|\partial_t\tilde\Theta_i\|_{C_tC_x^2}&\le C.
\label{eq:module-uniform}
\end{align}
The scalar initially has zero integral and unit $L^2((0,1)^2)$ norm. By using this construction, we establish boundary tangency, conservation for all module times, and smooth patching across interfaces in Section~\ref{sec:patching}.
\end{lemma}

For the following calculations, we omit the index $i$ and set $a={l}(1+\kappa h)$,
$b_1=\partial_s h$, and $q=(a^2+b_1^2)^{1/2}$.
Using \eqref{eq:frenet}, we obtain
\begin{equation}\label{eq:graph-tangent}
\partial_s\tilde{\gamma}=a\tau+b_1\eta,\qquad
\partial_s^2\tilde{\gamma}
 =(a_s+{l}\kappa b_1)\tau+(b_{1,s}-{l}\kappa a)\eta.
\end{equation}
Thus $q=|\partial_s\tilde\gamma|$. The stretching factor $\mathcal J_i=q_i/l$ has the second-order expansion
\begin{equation}\label{2-13}
\mathcal J_i=1+\kappa_i h_i+\frac12\left(\frac{\partial_s h_i}{l}\right)^2
 +O(\|h_i\|_{C_s^1}^3).
\end{equation}
For the normal graph $\tilde\gamma$, the unit tangent, normal and curvature are given by
\begin{align}
\widehat\tau&=q^{-1}(a\tau+b_1\eta),&
\widehat\eta&=q^{-1}(a\eta-b_1\tau),\label{2-16}\\
\widehat\kappa
 &=-\frac{\det(\tilde{\gamma}_s,\tilde{\gamma}_{ss})}{q^3}
 =\frac{{l}\kappa q^2+b_1a_s-a b_{1,s}}{q^3}.
\label{eq:exact-curvature}
\end{align}
For small $\epsilon_0$, $q\ge\frac12\min_t{l}(t)>0$. At each fixed time, the speed and curvature to first order are
\begin{align}
q&={l}+{l}\kappa h+O(\Norm{h}_{C_s^1}^2),\label{eq:length-linearization}\\
\widehat\kappa
 &=\kappa-\kappa^2h-{l}^{-2}h_{ss}+O(\Norm{h}_{C_s^2}^2).
\label{2-17}
\end{align}
The remainders are bounded in $C_s^0$, uniformly in time and in the admissible perturbation.

\subsection{Geometric conditions for the perturbation}
To construct a divergence-free velocity field compatible with $\tilde{\Gamma}_i$ and ensure smooth global patching, we impose an area-preserving constraint (a necessary condition) and an equal-length condition (a constructive condition). This subsection derives their concrete form.
\subsubsection{Area-preserving condition}
For the perturbed curves $\tilde{\Gamma}_i(t)$ to be compatible with a divergence-free velocity field tangent to the boundary of the square, the areas of the components they enclose with the boundary must remain constant.
\begin{lemma}\label{2.3}
At each fixed time, close the interior arc by a fixed boundary arc and orient the resulting closed curve by following $\gamma$ from $s=0$ to $s=1$. The perturbed signed area minus the reference signed area is
\begin{equation}\label{2-19}
\Delta A(h;t)
 =-\int_0^1{l}(t)\left(h+\frac12\kappa h^2\right)\,\mathrm d s.
\end{equation}
Condition (iv) preserves the areas of the two components of the square cut by each curve.
\end{lemma}
\begin{proof}
The fixed closing arc cancels when the signed area integrals are subtracted. Using $d=h\eta$, $d(0)=d(1)=0$ and integration by parts, we obtain
\begin{align*}
\Delta A
 &=\frac12\int_0^1\bigl[\det(d,\gamma_s)+\det(\gamma,d_s)+\det(d,d_s)\bigr]\,\mathrm d s\\
 &=\int_0^1\det(d,\gamma_s)\,\mathrm d s
   +\frac12\int_0^1\det(d,d_s)\,\mathrm d s.
\end{align*}
Using the Frenet formulas and $\det(\tau,\eta)=1$, we obtain
\[
\det(d,\gamma_s)=-{l}h,\qquad
\det(d,d_s)=h^2\det(\eta,\eta_s)=-{l}\kappa h^2,
\]
which proves \eqref{2-19}. Smallness preserves embeddedness and the prescribed boundary intersections. Hence signed-area preservation is also preservation of the corresponding component areas.
\end{proof}

We now prove the existence of nontrivial perturbation functions satisfying the area-preserving condition.
\begin{lemma}\label{2.4}
If each $K_i^{\mathrm{par}}$ has nonempty interior, there exists a nontrivial family $\{h_i\}$ satisfying conditions (i)--(iii) and the area-preserving condition (iv).
\end{lemma}
\begin{proof}
Fix a compact time interval $I\subset(0,1)$ of positive length. Choose compact spatial supports inside the interiors of $K_i^{\mathrm{par}}$, each with nonempty interior. Set
\[
X_i=C(I;C^6([0,1]))\cap C^1(I;C^5([0,1]))
\cap C^2(I;C^4([0,1])),
\]
with these spatial supports and the norm $\N_{2,6}$ on $I\times[0,1]$, making $X_i$ a Banach space. Define $F_i(h_i)(t)=\mathcal B_i(h;t)$. Then $F_i(0)=0$ and $F_i:X_i\to C^2(I)$ is smooth. The Fr\'echet derivative at $0$ is
\[
(DF_i(0)\varphi_i)(t)=\int_0^1 l(t)\varphi_i(t,s)\,ds.
\]
Choose a smooth supported $e_i$ with $\int_0^1le_i\,ds=1$. Then $a(t)\mapsto a(t)e_i(t,s)$ is a bounded right inverse, so $DF_i(0)$ is surjective.

Decompose $h_i=\alpha_i e_i+u_i$ with $u_i\in\ker DF_i(0)$ and $\alpha_i\in C^2(I)$. Define $G_i(\alpha_i,u_i)=F_i(\alpha_i e_i+u_i)$. Then $G_i(0,0)=0$ and $D_{\alpha_i}G_i(0,0)=\operatorname{id}$. By the implicit function theorem \cite[Section~4.7]{zeidler1993nonlinear}, there is a unique smooth map $u_i\mapsto\alpha_i(u_i)\in C^2(I)$ near zero such that $F_i(\alpha_i(u_i)e_i+u_i)=0$ and $\alpha_i(0)=D\alpha_i(0)=0$.

For smooth $u_i$ supported in the interior of $I$ in time, the implicit function theorem with parameter $t$ gives smoothness of $\alpha_i(t)$. By uniqueness, $\alpha_i(t)=0$ whenever $u_i(t,\cdot)=0$, so $h_i$ extends smoothly by zero to $[0,1]$. For nontriviality, if $h_i=0$, applying $DF_i(0)$ gives $\alpha_i=0$ and hence $u_i=0$. Thus $u_i\ne0$ gives $h_i\ne0$. As $\|u_i\|_{X_i}\to0$, continuity of $\alpha_i$ gives $\|h_i\|_{X_i}\to0$. Smooth zero-mean spatial functions multiplied by a time cutoff give infinitely many independent kernel directions, and hence an infinite-dimensional family.
\end{proof}

\subsubsection{Equal-length condition}
In the quasi-self-similar construction, smoothness of $v_n$ across adjacent blocks requires compatible boundary fields \cite{alberti2019exponential}. Under the area constraint and fixed endpoint geometry, a sufficient condition is that all perturbed central curves inside $[0,1]^2$ have equal total length. With $q_i=|\partial_s\tilde\gamma_i|$, the total length of a perturbed central curve is given by
\begin{equation}\label{2-20}
\tilde L_i(t)=\int_0^1q_i(t,s)\,ds
=\int_0^1\sqrt{l^2(1+\kappa_i h_i)^2+(\partial_s h_i)^2}\,ds.
\end{equation}
We therefore seek a nontrivial family $\{h_i\}_{i=1}^2$ satisfying both the area-preserving and equal-length conditions
\begin{equation}\label{2-21}
\mathcal B_1(h;t)=\mathcal B_2(h;t)=0,\qquad
\tilde L_1(t)=\tilde L_2(t).
\end{equation}
The remaining configurations inherit these conditions by rotation.
\begin{lemma}\label{2.5}
Under the support assumptions of Lemma~\ref{2.4}, suppose that the curvature $\kappa_2(t_0,\cdot)$ of the second generator is nonconstant in the interior of its allowed support region $K_2^{\mathrm{par}}$ for some $t_0\in(0,1)$. Then there is an infinite-dimensional family of nonzero smooth perturbations satisfying \eqref{2-21}, with arbitrarily small norm \eqref{eq:smallness}. Their supports lie in fixed compact subsets of the interiors of $K_i^{\mathrm{par}}$ and a fixed compact time interval around $t_0$ contained in $(0,1)$.
\end{lemma}
\begin{proof}
Since $\kappa_2(t_0,\cdot)$ is nonconstant, choose a smooth function $q_*(s)$ supported in the interior of $K_2^{\mathrm{par}}$ such that
\[
\int_0^1q_*(s)\,ds=0,\qquad
c(t):=\int_0^1l(t)\kappa_2(t,s)q_*(s)\,ds,\qquad c(t_0)\ne0.
\]
Take the difference of two equal-integral bump functions with sufficiently small supports near points of distinct curvature. Then $\int l(t)q_*=0$ for every $t$. Choose a closed interval $I\subset(0,1)$ around $t_0$ with $|c(t)|\ge c_0>0$, using the continuity of $c$. Choose compact supports in the interiors of $K_i^{\mathrm{par}}$, with nonempty interiors and the second containing $\operatorname{supp}q_*$. On these supports, choose smooth $e_i(t,s)$ with $\int le_i=1$. Then we define
\[
F_t(h)=\bigl(\mathcal B_1(h;t),\mathcal B_2(h;t),\tilde L_2(t)-\tilde L_1(t)\bigr).
\]
Its derivative at zero is
\begin{equation}\label{eq:constraint-linearization}
T_th=\left(\int{l} h_1,\int{l} h_2,
 \int{l}\kappa_2h_2-\int{l}\kappa_1h_1\right).
\end{equation}
Equation~\eqref{eq:constraint-linearization} gives $T_t(0,q_*)=(0,0,c(t))$, so the curvature condition on one generator suffices. The images of $(e_1,0)$, $(0,e_2)$ and $(0,q_*)$ are linearly independent because $c(t)\ne0$. Define $R_t:\R^3\to C_c^\infty((0,1))^2$ by inverting $T_t$ on their span. Then $T_tR_t=\operatorname{id}$, and $R_t$ depends smoothly on $t$.

As in Lemma~\ref{2.4}, let $X=X_1\times X_2$ be the product space with the chosen spatial supports, and set $Y=C^2(I;\R^3)$. The area terms are polynomial. The square root in the length formula has a uniformly positive argument. Integration in $s$ is continuous. Thus $F:X\to Y$ is smooth near zero. The unperturbed curves have equal length and zero area difference, so $F(0)=0$. Here $(Th)(t)=T_th$ and $DF(0)=T$. Uniform bounds for $R_t$ and its first two time derivatives give a bounded operator $R:Y\to X$ with $TR=\operatorname{id}_Y$.

Define $G(u,a)=F(u+Ra)$ for $u$ in the closed subspace $\ker T$ and $a\in Y$. Then $G(0,0)=0$ and $D_aG(0,0)=\operatorname{id}_Y$. By the implicit function theorem \cite[Section~4.7]{zeidler1993nonlinear}, there is a unique smooth map $B$ near zero such that
\begin{equation}\label{eq:constraint-solution}
\begin{aligned}
&h=u+RB(u),\qquad F(h)=0,\\
&B(0)=0,\qquad DB(0)=0,\\
&\Norm{RB(u)}_X\le C\Norm{u}_X^2.
\end{aligned}
\end{equation}%
\indent Take smooth $u\in\ker T$ with time support compactly contained in the interior of $I$. At each fixed time, the constraint is a system of three equations for $a(t)$. Its Jacobian with respect to $a(t)$ is the identity at $(u,a)=(0,0)$ and remains invertible near zero. The implicit function theorem with parameter $t$ therefore shows that $a(t)=B(u)(t)$ is smooth. Pointwise uniqueness near zero gives $a(t)=0$ whenever $u(t)=0$. Thus $h$ extends smoothly by zero to $[0,1]$, has the prescribed spatial support, and vanishes near both time endpoints.

Since $u=h-RTh$, different kernel elements give different perturbations, and $u\ne0$ gives $h\ne0$. The space of smooth supported spatial pairs is infinite-dimensional, and $\ker T_{t_0}$ has codimension three. Choose linearly independent pairs $w_j$ in this kernel and a smooth time cutoff $\chi$ supported in the interior of $I$, with $\chi(t_0)=1$. Set
\[
u_j(t)=\chi(t)(\operatorname{id}-R_tT_t)w_j.
\]
Then $Tu_j=0$ and $u_j(t_0)=w_j$, so the functions $u_j$ are linearly independent. Applying \eqref{eq:constraint-solution} to sufficiently small finite linear combinations gives the required infinite-dimensional family. As in Lemma~\ref{2.4}, smallness follows by taking the kernel elements sufficiently small.
\end{proof}

The equal-length condition is a convenient sufficient condition for compatible boundary fields under the area-preserving condition and fixed endpoint geometry. The second (bent) generator in \cite[Section~8.2(a)--(b)]{alberti2019exponential} has different endpoint tangents and is straight near its endpoints. Thus, at any fixed $t_0\in(0,1)$, its curvature is nonzero at one interior point and zero at another. Choose compact support regions with nonempty interiors, with both points in the interior of $K_2^{\mathrm{par}}$, leaving smaller endpoint neighbourhoods unchanged. Lemma~\ref{2.5} then applies on a small interior time interval around $t_0$.

\subsection{Construction and stability of the perturbed area-preserving map}
For any perturbation satisfying (i)--(v), in this subsection we construct the perturbed area-preserving maps $\tilde{\Phi}_i$ compatible with $\tilde{\Gamma}_i$ and prove their stability together with that of their inverses, completing Lemma~\ref{2.1}.
\subsubsection{Definition of the perturbed area-preserving map}\label{sec:coordinates}
Let $\tilde{l}(t)=\tilde L_1(t)=\tilde L_2(t)$ and define $S_i(t,s)$, for $s\in\R$, by
\begin{equation}\label{eq:reparametrisation}
\int_0^{S_i(t,s)}q_i(t,\xi)\,\mathrm d\xi=\tilde{l}(t)s.
\end{equation}
The extension uses $q_i={l}$ outside $[0,1]$. Since $q_i$ has a uniform positive lower bound, this defines a unique smooth $S_i$.
Set
\begin{equation}\label{eq:constant-speed}
\tilde{\gamma}_i^{\mathrm{arc}}(t,s)=\tilde{\gamma}_i(t,S_i(t,s)).
\end{equation}
By \eqref{eq:reparametrisation}, $S_i(t,0)=0$, $S_i(t,1)=1$ and
$\partial_sS_i=\tilde{l}/q_i(t,S_i)>0$. Thus the endpoints remain fixed and
$|\partial_s\tilde{\gamma}_i^{\mathrm{arc}}|=\tilde{l}$. In a straight endpoint neighbourhood, the parametrisation is the fixed endpoint plus $\tilde{l}$ times the shifted parameter in the fixed tangent direction. It therefore extends smoothly to the prescribed rays. Differentiation in time gives
\begin{equation}\label{eq:reparam-time}
\partial_t\tilde{\gamma}_i^{\mathrm{arc}}
 =\tilde{\gamma}_{i,t}\circ S_i
  +(\tilde{\gamma}_{i,s}\circ S_i)\partial_tS_i.
\end{equation}

On a fixed compact parameter interval containing the non-straight parts of the reference curves, their uniformly positive speed and $C^2$ bounds give injective normal-coordinate charts of a common size. For parameter pairs separated by at least this size, embeddedness and compactness in time give a positive distance between the corresponding reference points. These properties persist for sufficiently small $C_s^2$ perturbations. The fixed straight rays are separated from nonadjacent pieces, so a smaller width excludes intersections there as well. Thus the same tube radius works for all perturbed curves and is unchanged by reparametrisation.

Outside fixed endpoint pieces, the interior arcs remain uniformly separated from $\partial[0,1]^2$. Small perturbations preserve this separation, while the endpoint pieces remain unchanged. Using \eqref{eq:beta}, the uniform curvature bound and the uniform upper and lower bounds for $\tilde l$, we choose $r>0$, independently of $t$ and $h$, so that the maps \eqref{2-8} are injective and their tubes meet $\partial[0,1]^2$ only in the prescribed straight endpoint rectangles.

The unit tangent, curvature and unit normal of $\tilde\gamma_i^{\mathrm{arc}}$ are
\begin{equation}\label{eq:geometry-composed}
\tilde{\tau}_i=\widehat\tau_i\circ S_i,\quad \tilde{\kappa}_i=\widehat\kappa_i\circ S_i,
\qquad \tilde{\eta}_i=\widehat\eta_i\circ S_i.
\end{equation}
Formula \eqref{eq:beta} defines $\tilde\beta_i$, which is smooth even where the curvature vanishes and satisfies
\begin{equation}\label{eq:area-identity}
\tilde{\beta}_i+\frac12\tilde{\kappa}_i(\tilde{\beta}_i)^2=\frac y{\tilde{l}},
\qquad \tilde{l}(1+\tilde{\kappa}_i\tilde{\beta}_i)\partial_y\tilde{\beta}_i=1.
\end{equation}
In the oriented frame $(\tilde{\tau}_i,\tilde{\eta}_i)$, the two columns of the spatial differential of $\tilde{\Phi}_i$ are
$\bigl(\tilde{l}(1+\tilde{\kappa}_i\tilde{\beta}_i),\partial_s\tilde{\beta}_i\bigr)^T$
and $(0,\partial_y\tilde{\beta}_i)^T$. Hence
\begin{equation}\label{eq:jacobian}
\det D_{s,y}\tilde{\Phi}_i=1.
\end{equation}
The inverse is smooth by the inverse function theorem and the preceding global injectivity argument.

\subsubsection{Stability of the area-preserving maps and inverses}\label{sec:map-estimates}
To estimate $\tilde{\Phi}_i-\Phi_i$ and its derivatives, we decompose the difference as
\[
\tilde{\Phi}_i-\Phi_i
 =(\tilde{\gamma}_i^{\mathrm{arc}}-\gamma_i)
  +(\tilde{\beta}_i-\beta_i)\tilde{\eta}_i
  +\beta_i(\tilde{\eta}_i-\eta_i).
\]
We first establish stability estimates for the geometric quantities on fixed compact parameter intervals. The product and composition rules for \eqref{eq:triangular-norm} are proved in Lemma~\ref{lem:calculus}. By \eqref{eq:graph-tangent}, $q_i$ depends on only one spatial derivative of $h_i$, so
\[
\N_{2,5}(q_i-l)\le C\epsilon,\qquad
\Norm{\tilde{l}-{l}}_{C_t^2}\le C\epsilon.
\]
From \eqref{eq:reparametrisation}, we obtain
\[
l(t)(S_i(t,s)-s)
 =(\tilde{l}(t)-l(t))s
  -\int_0^{S_i(t,s)}(q_i(t,\xi)-l(t))\,\mathrm d\xi.
\]
On fixed compact parameter intervals, the positive lower bound for $q_i$ bounds $S_i$, so this identity gives the $C^0$ estimate. We differentiate this identity up to total order five, with at most two derivatives in time. In each resulting equation, the coefficient of the highest-order derivative of $S_i-\operatorname{id}$ is $q_i(t,S_i)\ge c>0$. Induction therefore yields
\begin{equation}\label{eq:sigma-estimate}
\N_{2,5}(S_i-\operatorname{id})\le C\epsilon,
\end{equation}
with uniform bounds for $S_i$ on fixed compact parameter intervals.

Formula \eqref{eq:exact-curvature} involves spatial derivatives of $h_i$ up to order two. We first estimate this expression before composing with $S_i$. Using \eqref{2-16}, \eqref{eq:constant-speed}, and \eqref{eq:geometry-composed}, we obtain
\begin{align}
\N_{2,4}(\tilde{\kappa}_i-\kappa_i)&\le C\epsilon,\label{eq:curvature-estimate}\\
\N_{2,5}(\tilde{\eta}_i-\eta_i)
 +\N_{2,5}(\tilde{\gamma}_i^{\mathrm{arc}}-\gamma_i)&\le C\epsilon.
\label{eq:normal-estimate}
\end{align}
All corresponding individual norms are uniformly bounded. Substituting these estimates into the smooth formula \eqref{eq:beta}, whose denominators stay uniformly positive, and using the decomposition above, we obtain
$\N_{2,4}(\tilde{\Phi}_i-\Phi_i)\le C\epsilon$ on fixed compact rectangles inside $D$. This proves the forward estimates in \eqref{2-9} and \eqref{eq:map-uniform}.

To compare the inverse maps on the moving tubes, we omit the index $i$ and choose bounded open rectangles $D_1,D_2$ with $\overline{D_1}\subset D_2$ and $\overline{D_2}\subset D$, with half-widths in $y$ larger than $3r/4$ and smaller than $r$, and endpoints in $s$ sufficiently far out on the fixed rays.
Choose $D_1$ to contain all preimages of the velocity, scalar and truncated potential supports in $\overline W$, for both the reference and perturbed constructions. Since the cutoff is supported in $|y|\le3r/4$ and the parametrisations and their straight rays satisfy uniform bounds, this choice is possible. By compactness in time and injectivity of the reference maps, there exists $d>0$ such that
\[
\dist\bigl(\Phi(t,\overline D_1),\Phi(t,\partial D_2)\bigr)\ge3d
\qquad\text{for every }t\in[0,1].
\]
By the forward estimate, decrease $\epsilon_0$ so that
$\Norm{\tilde\Phi-\Phi}_{C([0,1]\times\overline D_2)}<d$.
On $D_2$, the homotopy
$\Phi+\lambda(\tilde{\Phi}-\Phi)$ maps the boundary away from
$\Phi(t,\overline D_1)$. Its degree is one, which gives the first inclusion below:
\[
\Phi(t,\overline D_1)\subset\tilde\Phi(t,D_2),
\qquad
\tilde\Phi(t,\overline D_1)\subset\Phi(t,D_2).
\]
If $z\in\overline D_1$, the segment from $\Phi(t,z)$ to
$\tilde\Phi(t,z)$ has length less than $d$ and stays inside
$\Phi(t,D_2)$, since its starting point has distance at least $3d$
from the boundary. This proves the second inclusion.

Since the reference separation is at least $3d$ and each perturbation is less than $d$, the relevant supports in $\overline W$ lie in both images of $D_2$, at distance at least $d$ from each boundary, uniformly in $t$ and $h$. Both inverses are defined on the common set
\[
\mathcal U_*=\{(t,x):0\le t\le1,\ 
x\in\Phi(t,D_2)\cap\tilde\Phi(t,D_2)\}.
\]
Both inverse images lie in $D_2$, with $\overline{D_2}\subset D$. On a fixed larger compact rectangle in $D$, the forward bounds and \eqref{eq:jacobian} bound the inverse differentials uniformly by Lemma~\ref{A.2}. For $(t,x)\in\mathcal U_*$, we have
\[
|\Phi(t,\tilde\Psi(t,x))-\Phi(t,\Psi(t,x))|
=|(\Phi-\tilde\Phi)(t,\tilde\Psi(t,x))|\le C\epsilon.
\]
Injectivity of the reference maps and compactness of $[0,1]\times\overline D_2$ imply that the two preimages are uniformly close for small $\epsilon$. The parameter margin, the fixed reference $C^2$ bounds and uniform $C^1$ closeness then give common uniform inverse-function charts containing both preimages. Lemma~\ref{lem:inverse-calculus} therefore yields
$\Norm{\tilde\Psi-\Psi}_{C(\mathcal U_*)}\le C\epsilon$.
Differentiating $\tilde{\Phi}_i(t,\tilde{\Psi}_i(t,x))=x$ in space and time, we obtain
\begin{equation}\label{eq:inverse-derivatives}
D_x\tilde{\Psi}_i=(D_z\tilde{\Phi}_i\circ\tilde{\Psi}_i)^{-1},\qquad
\partial_t\tilde{\Psi}_i=-(D_z\tilde{\Phi}_i\circ\tilde{\Psi}_i)^{-1}
 (\partial_t\tilde{\Phi}_i\circ\tilde{\Psi}_i).
\end{equation}
Lemma~\ref{lem:inverse-calculus} gives the inverse estimates in \eqref{2-9} and \eqref{eq:map-uniform}. The additional derivatives used to compare compositions belong only to the fixed smooth reference. This proves Lemma~\ref{2.1}.

\subsection{Stability of the local velocity and scalar fields}\label{sec:complete-fields}
\begin{proof}[Proof of Lemma~\ref{2.2}]
For a fixed generator, we omit the index and define the velocity on its tube by
\begin{equation}\label{eq:local-velocity}
\tilde{w}(t,x)=\partial_t\tilde{\Phi}(t,\tilde{\Psi}(t,x)).
\end{equation}
The Jacobian identity and \eqref{eq:jacobian} give
\[
0=\partial_t\det D_z\tilde{\Phi}
 =(\operatorname{div} \tilde{w})\circ\tilde{\Phi}\det D_z\tilde{\Phi},
\]
so $\tilde{w}$ is divergence-free. Using a stream function, we construct a divergence-free field that agrees with $\tilde w$ where the cutoff is identically one, as in \cite[Proposition 7.1]{alberti2019exponential}.

On the simply connected strip $D$, we define
\begin{equation}\label{eq:one-form}
\tilde{a}=-(D_z\tilde{\Phi})^T J\partial_t\tilde{\Phi}.
\end{equation}
Differentiating \eqref{eq:one-form} and using \eqref{eq:jacobian}, we obtain $\partial_s\tilde a_2-\partial_y\tilde a_1=-\partial_t\det D_z\tilde\Phi=0$.
Define the normalised potential by
\begin{equation}\label{eq:potential}
\tilde{P}(t,s,y)=\int_0^s \tilde{a}_1(t,\xi,0)\,\mathrm d\xi
 +\int_0^y \tilde{a}_2(t,s,\zeta)\,\mathrm d\zeta.
\end{equation}
We obtain $\nabla_z\tilde{P}=\tilde{a}$ and $\tilde{P}(t,0,0)=0$.
Thus $\tilde{\psi}=\tilde{P}\circ\tilde{\Psi}$ satisfies
$\nabla_x\tilde{\psi}=-J\tilde{w}$ and $\tilde{w}=J\nabla_x\tilde{\psi}$.

Using the fixed cutoff $g$ from the statement, we define
\begin{equation}\label{eq:complete-potential}
\tilde{\varphi}(t,x)=
\begin{cases}
(g\tilde{P})(t,\tilde{\Psi}(t,x)),&x\in\tilde{\Phi}(t,D),\\
0,&x\notin\tilde{\Phi}(t,D),
\end{cases}
\qquad \tilde{V}=J\nabla_x\tilde{\varphi}.
\end{equation}
The cutoff $g$ depends on the transverse parameter $y$ and vanishes near the boundary of the parameter strip, so the zero extension is smooth. On the tube,
\begin{equation}\label{eq:complete-velocity}
\tilde{V}=(g\circ\tilde{\Psi})\tilde{w}
 +(\tilde{P}\circ\tilde{\Psi})J\nabla_x(g\circ\tilde{\Psi}).
\end{equation}
By construction, $\operatorname{div}\tilde V=0$ everywhere.

Define $\tilde{\Theta}(\tilde{\Phi}(t,s,y),t)=c_\Theta \bar{\Theta}(y/r)$ and extend it by zero. Its support is inside the region where $g=1$, and then $\tilde{V}=\tilde{w}$. The scalar is constant on each trajectory $t\mapsto\tilde{\Phi}(t,s,y)$. After smooth zero extension, the transport equation holds everywhere. At the initial time, the part of the tube core inside $(0,1)^2$ corresponds to $0<s<1$, $|y|<r/2$. Indeed, the tube width ensures that the interior tube avoids the other sides, while the normal coordinate in each straight endpoint map is parallel to the corresponding side. Since the Jacobian is one,
\[
\int_{(0,1)^2}\tilde{\Theta}(0)=c_\Theta r\int \bar{\Theta}=0,\qquad
\int_{(0,1)^2}|\tilde{\Theta}(0)|^2=c_\Theta^2r\int \bar{\Theta}^2=1.
\]

We first estimate the pullback quantities on a fixed compact parameter rectangle. Formula \eqref{eq:one-form} and the product rule, together with the $\N_{2,4}$ bounds in Lemma~\ref{2.1}, give
\[
\N_{1,3}(\tilde a-a)\le C\epsilon,\qquad
\N_{1,3}(\tilde a)\le C.
\]
For $a\le1$ and $a+|\beta|\le3$, applying $\partial_t^aD_z^\beta$ to $\partial_t\tilde\Phi$ and $D_z\tilde\Phi$ requires at most two and one time derivatives, respectively, and total order at most four in both cases. Using Lemma~\ref{2.1}, we obtain uniform bounds for both factors. Expanding the product difference yields the factor $\epsilon$.

The identity $\nabla_z\tilde P=\tilde a$ gives one more spatial derivative. If $|\beta|\ge1$, choose $j$ with $\beta_j\ge1$. For $a\le1$ and $a+|\beta|\le4$,
\[
\partial_t^aD_z^\beta\tilde P
=\partial_t^aD_z^{\beta-e_j}\tilde a_j,
\qquad a+|\beta-e_j|\le3.
\]
The same identity holds for the difference $\tilde P-P$. For the remaining derivatives, $\beta=0$ and $a=0,1$, we differentiate the normalised integral formula \eqref{eq:potential}. Since the integration paths lie in a fixed compact rectangle and have uniformly bounded length, we obtain
\begin{equation}\label{eq:potential-estimates}
\N_{1,4}(\tilde P-P)\le C\N_{1,3}(\tilde a-a)\le C\epsilon,
\qquad \N_{1,4}(\tilde P)\le C.
\end{equation}

On the common spacetime set $\mathcal U_*$, we decompose the difference of the truncated physical potentials as
\begin{equation}\label{eq:composition-split}
\tilde\varphi-\varphi
=[g(\tilde P-P)]\circ\tilde\Psi
+(gP)\circ\tilde\Psi-(gP)\circ\Psi.
\end{equation}
The first term is bounded using \eqref{eq:potential-estimates} and the uniform inverse bounds. For the second term, we apply the mean-value formula in the common parameter rectangle and then differentiate. The additional spatial derivative belongs only to the fixed smooth function $gP$. Since the $\N_{2,4}$ inverse estimates control $\N_{1,4}$, Lemma~\ref{lem:calculus} yields
\[
\N_{1,4}(\tilde\varphi-\varphi)\le C\epsilon,
\qquad\N_{1,4}(\tilde\varphi)\le C.
\]
Differentiating once in $x$ gives
\begin{equation}\label{eq:velocity-budget}
\N_{1,3}(\tilde V-V)\le C\epsilon,
\qquad\N_{1,3}(\tilde V)\le C.
\end{equation}
We now prove the scalar estimates in Lemma~\ref{2.2}. By definition, the scalar is a fixed smooth template composed with $\tilde\Psi$, so the inverse-map estimates give the same bounds. The union of the relevant supports in $\overline W$ has the uniform spatial margin proved above inside $\mathcal U_*$. Within $W$, outside $\mathcal U_*$, both truncated potentials and both scalars vanish on open neighbourhoods. Thus the estimates hold on the whole square, including the symmetric difference of the moving tubes.

This proves \eqref{2-11} and \eqref{eq:module-uniform}, with the derivative count
\begin{equation}\label{eq:derivative-budget}
h:\N_{2,6},\quad
(\Phi,\Psi):\N_{2,4},\quad
a:\N_{1,3},\quad
(P,\varphi):\N_{1,4},\quad V:\N_{1,3}.
\end{equation}
Thus no derivatives of $h$ beyond \eqref{eq:smallness} enter the uniform constants.
\end{proof}

Applying the area-preserving and equal-length conditions, we have proved the mixed stability estimates for the area-preserving maps and the spatial $C^3$ stability of the full local fields, together with spatial $C^2$ stability of their first time derivatives. These provide the geometric foundation for Sections~3 and~4.

\section{Perturbed quasi-self-similar family: construction and estimates}\label{sec:patching}

This section proves Lemma~\ref{1.1} and Theorem~\ref{1.2}. Patching the perturbed fields $(\tilde V_i,\tilde\Theta_i)$ from Section~2 via the Peano curve yields a global quasi-self-similar family $\{\tilde\rho_n,\tilde v_n\}$ that retains gradient growth, recursion, and compact support.

\subsection{Construction of the perturbed basic family and the global quasi-self-similar family}
\par
We first complete the proof of Lemma \ref{1.1}.

\begin{proof}[Proof of Lemma \ref{1.1}]
We compute the full local field near each endpoint and verify the hypotheses of \cite[Lemma 7.8]{alberti2019exponential}.

Let $x_0$ be a fixed entry or exit point, let $\tau$ be the constant oriented unit tangent, and set $\eta=J\tau$. For this local formula, shift the parameter so that $s=0$ at $x_0$. In the straight region, the coordinate map is
\begin{equation}\label{eq:interface-map}
\tilde{\Phi}(t,s,y)=x_0+\tilde{l}(t)s\tau+\frac{y}{\tilde{l}(t)}\eta.
\end{equation}
Write $x-x_0=\xi\tau+\zeta\eta$ in the basis $(\tau,\eta)$, where $\xi=(x-x_0)\cdot\tau=\tilde l(t)s$ and $\zeta=(x-x_0)\cdot\eta=y/\tilde l(t)$. Set $a_h(t)=\tilde l'(t)/\tilde l(t)$. Differentiating at fixed $(s,y)$ and using $\tilde w=J\nabla_x\tilde\psi$, we obtain
\begin{equation}\label{eq:interface-velocity}
\tilde{w}=a_h\xi\tau-a_h\zeta\eta,
\qquad \tilde{\psi}=-a_h\xi\zeta+C_h(t).
\end{equation}
On the full central arc, parametrised by $s\in[0,1]$, the potential is zero at $s=0$ and satisfies
\begin{equation}\label{eq:potential-flux}
\frac{\mathrm d}{\mathrm ds}\tilde{\psi}(t,\tilde{\gamma}^{\mathrm{arc}}(t,s))
 =\tilde{l}(t)\tilde{w}(t,\tilde{\gamma}^{\mathrm{arc}}(t,s))\cdot\tilde{\eta}(t,s).
\end{equation}
By Lemma~\ref{2.3} and \eqref{2-21}, the area enclosed by the central arc and the fixed boundary arc is constant. Differentiating its signed-area formula in time gives
\[
\int_0^1\tilde{l} \tilde{w}(t,\tilde{\gamma}^{\mathrm{arc}}(t,s))\cdot\tilde{\eta}(t,s)\,\mathrm d s=0.
\]
Integrating \eqref{eq:potential-flux} and using the entry normalisation gives $C_h(t)=0$ at both endpoints. Hence the cutoff potential and scalar in the straight region are
\begin{equation}\label{eq:interface-complete}
\tilde{\varphi}=-a_h\xi\zeta\,g(\tilde{l}\zeta),
\qquad \tilde{\Theta}=c_\Theta \bar{\Theta}(\tilde{l}\zeta/r).
\end{equation}
On the side through $x_0$, $\xi=0$ and $\tau$ is normal to the side. Thus $\tilde V\cdot\tau=-\partial_\zeta\tilde\varphi=0$. Outside the endpoint rectangles, both fields vanish near the boundary by the choice of tube width. Hence $\tilde V$ is tangent to the whole boundary, including the corners.

By the fixed patching rule, adjacent modules crossing an interface have the same crossing point and straight line. They also share $\tilde l$, $a_h$, cutoffs and normalisations. Reversing the orientation sends $(\xi,\zeta)$ to $(-\xi,-\zeta)$. Since $\xi\zeta$ is unchanged and $g,\bar\Theta$ are even, both expressions in \eqref{eq:interface-complete} are unchanged.

Rotations commute with $J$, so the same argument applies to all six configurations. Thus the formulas \eqref{eq:interface-complete} agree on an open rectangle across each crossed interface. Both fields vanish near the remaining interfaces. These identities hold for every module time, so all spatial and mixed time-space derivatives agree.

By integration by parts on $(0,1)^2$, boundary tangency and the transport equation now give
\[
\frac{\mathrm d}{\mathrm dt}\int_{(0,1)^2}\tilde{\Theta}_i=0,
\qquad \frac{\mathrm d}{\mathrm dt}\int_{(0,1)^2}|\tilde{\Theta}_i|^2=0.
\]
The initial normalisation in Lemma~\ref{2.2} proves \eqref{1.2-2} for every $t$.

Let $\lambda=1/5$. Since $h(0,\cdot)=h(1,\cdot)=0$, the initial and terminal curves, their constant-speed parametrisations, and their coordinate maps are thus the reference ones. All spatial derivatives of $h$ at these two times vanish as well, since each endpoint value is the zero function of $s$. The scalar uses the same template and normalisation, so we have
\begin{equation}\label{eq:endpoint-scalars}
\tilde\Theta_i(\cdot,0)=\Theta_i(\cdot,0),\qquad
\tilde\Theta_i(\cdot,1)=\Theta_i(\cdot,1).
\end{equation} The terminal curve consists of twenty-five scaled initial arcs, giving
\begin{equation}\label{eq:length-recursion}
\tilde{l}(1)=25\lambda\tilde{l}(0)=\lambda^{-1}\tilde{l}(0).
\end{equation}
The construction preserves the transverse area coordinate. Indeed, if $d_0$ and $\kappa_0$ are the signed normal distance and curvature for an initial arc, scaling yields $d_1=\lambda d_0$ and $\kappa_1=\lambda^{-1}\kappa_0$. By \eqref{eq:area-identity},
\begin{align}
y_1&=\tilde{l}(1)\left(d_1+\frac12\kappa_1d_1^2\right)\notag\\
 &=\tilde{l}(0)\left(d_0+\frac12\kappa_0d_0^2\right)=y_0.
\label{eq:transverse-recursion}
\end{align}
Orientation reversal changes $y_1$ to $-y_0$ and leaves the even template $\bar\Theta$ unchanged. With the same $r,c_\Theta$ and index $j=j(Q,i)$, the endpoint identities \eqref{eq:endpoint-scalars} and the reference recursion yield
\begin{align}
\tilde\Theta_i(x,1)&=\Theta_i(x,1)
=\Theta_j\bigl(5(x-r_Q),0\bigr)\notag\\
&=\tilde\Theta_j\bigl(5(x-r_Q),0\bigr),\qquad x\in Q.
\label{eq:inherited-recursion}
\end{align}
The identity extends to tile boundaries by smoothness. Together with Lemma~\ref{2.2} and the boundary and patching properties above, we establish Lemma~\ref{1.1}.

\end{proof}

We define the perturbed quasi-self-similar family $\{\tilde{\rho}_n(x,t),\tilde{v}_n(x,t)\}_{n\in\mathbb{N}}$ on $[0,1]^2\times[0,1]$. Denote by $\mathcal{Q}(2\cdot5^n)$ a partition of the unit square $[0,1]^2$ into closed squares of side length $(2\cdot5^n)^{-1}$ with disjoint interiors. For each $Q\in\mathcal{Q}(2\cdot5^n)$, let $r_Q$ be its lower-left corner, so that $Q=(2\cdot5^n)^{-1}Q_0+r_Q$ with $Q_0=[0,1]^2$. For every $n\in\mathbb{N}$ and $t\in[0,1]$, first define the fields on the tile interiors by
\begin{equation}
    \tilde{\rho}_n(x,t) = \sum_{Q\in\mathcal{Q}(2\cdot5^n)} \chi_Q(x)\,\tilde{\Theta}_{i(Q)}\bigl(2\cdot5^n(x-r_Q),\,t\bigr),
    \label{3-1}
\end{equation}
\begin{equation}
    \tilde{v}_n(x,t) = \sum_{Q\in\mathcal{Q}(2\cdot5^n)} \chi_Q(x)\,\frac{1}{2\cdot5^n}\,\tilde{V}_{i(Q)}\bigl(2\cdot5^n(x-r_Q),\,t\bigr),
    \label{3-2}
\end{equation}
where $i(Q)\in\{1,\dots,N\}$ is determined in \cite{alberti2019exponential} and $\chi_Q$ denotes the characteristic function of $Q$. The sums hold on tile interiors and almost everywhere. On tile boundaries, use the common smooth interface expressions, which vanish near the vertices.

The index $i(Q)$ includes the prescribed rotation. The initial four bent configurations form the annulus of \cite[Section 8.1]{alberti2019exponential}. The terminal substitutions determine the later assignments. Every interior crossing is paired and the outer boundary is uncrossed. The interface formulas agree, yielding globally smooth fields satisfying the transport equation. The divergence-free property of each $\tilde V_i$ ensures that $\operatorname{div}\tilde v_n=0$ globally.

\subsection{Estimates for the perturbed quasi-self-similar family}
We now prove Theorem \ref{1.2} using the stability estimates from Section~2. The constant $C$ depends only on the fixed data and is independent of $n$, $\epsilon$ and $t$.
\begin{proof}[Proof of Theorem \ref{1.2}]
For $(\tilde v_n,\tilde\rho_n)$ defined by \eqref{3-1}--\eqref{3-2}, the following properties hold for every $n\in\mathbb N$ and $t\in[0,1]$.

\noindent\textbf{Step 1. H\"older estimates for the velocity field.}
For any $Q\in\mathcal{Q}(2\cdot5^n)$, set $z = 2\cdot5^n(x-r_Q)\in[0,1]^2$. Then we have $\tilde{v}_n = \frac{1}{2\cdot5^n}\tilde{V}_{i(Q)}(z,t)$. For $k=0$ and $j=0,1,2,3$, or $k=1$ and $j=0,1,2$, Lemma~\ref{2.2} gives
\[
\|D_x^j\partial_t^k\tilde v_n\|_\infty\le C(2\cdot5^n)^{j-1},\qquad
\|D_x^j\partial_t^k(\tilde v_n-v_n)\|_\infty\le C\epsilon(2\cdot5^n)^{j-1}.
\]
These estimates hold globally because the derivatives agree across the interfaces. For a smooth periodic function $F$, we know
\begin{equation}\label{eq:holder-interpolation}
[F]_{C^\alpha}\le C(\alpha)\|F\|_\infty^{1-\alpha}\|\nabla F\|_\infty^\alpha,
\qquad 0<\alpha<1.
\end{equation}
We obtain this estimate by bounding $|F(x)-F(y)|$ by the smaller of $2\|F\|_\infty$ and $\|\nabla F\|_\infty\operatorname{dist}(x,y)$. Using $j=0,1$ proves \eqref{1.3-1}. Using $j=2,3$ gives
\[
\|\Delta\tilde v_n\|_{C^\alpha}\le C(\alpha)5^{(1+\alpha)n}.
\]
The same argument gives the two difference bounds with factor $\epsilon$.

\noindent\textbf{Step 2. Conservation properties.}
With the change of variables $z=2\cdot5^n(x-r_Q)$, we have $dx=(2\cdot5^n)^{-2}dz$. Hence
\[
\int_{[0,1]^2}\tilde{\rho}_n\,dx = \sum_{Q}(2\cdot5^n)^{-2}\int_{[0,1]^2}\tilde{\Theta}_{i(Q)}(z,t)\,dz = 0,
\]
since each $\tilde{\Theta}_{i(Q)}$ has zero mean. Since the interiors of the tiles are disjoint and their boundaries have zero measure,
\[
\int_{[0,1]^2}|\tilde{\rho}_n|^2\,dx = \sum_{Q}(2\cdot5^n)^{-2}\int_{[0,1]^2}|\tilde{\Theta}_{i(Q)}(z,t)|^2\,dz = (2\cdot5^n)^{-2}|\mathcal{Q}(2\cdot5^n)| = 1.
\]

\noindent\textbf{Step 3. $L^\infty$ and gradient estimates.}
From Lemma \ref{2.2} and the fixed reference fields, for all sufficiently small $\epsilon$, there exist constants $C_0,C_1>0$, independent of the admissible perturbation, such that
\[
\|\tilde{\Theta}_i\|_{L^\infty} \le C_0+C_1\epsilon,\quad
2\|\nabla\tilde{\Theta}_i\|_{L^\infty} \le C_0+C_1\epsilon.
\]
On each square $Q$, $\tilde{\rho}_n(x,t)=\tilde{\Theta}_{i(Q)}(z,t)$ and $\nabla_x\tilde{\rho}_n = 2\cdot5^n\nabla_z\tilde{\Theta}_{i(Q)}$. Therefore, we obtain
\[
\|\tilde{\rho}_n\|_{L^\infty}\le C_0+C_1\epsilon,\quad
\|\nabla\tilde{\rho}_n\|_{L^\infty}\le (C_0+C_1\epsilon)5^n.
\]

The same calculation gives zero mean on every tile. By the Poincar\'e inequality (Lemma~\ref{A.3}) on each tile,
\[
1=\sum_Q\|\tilde\rho_n\|_{L^2(Q)}^2
\le C(2\cdot5^n)^{-2}\sum_Q\|\nabla\tilde\rho_n\|_{L^2(Q)}^2.
\]
Together with the upper gradient estimate, this proves the two gradient bounds in \eqref{eq:scalar-mixing}. For a mean-zero test function $g\in H^1(\mathbb T^2)$, subtracting its tile averages $\bar g_Q$ and using Cauchy--Schwarz and Poincar\'e, we get
\[
\left|\int\tilde\rho_n g\right|
=\left|\sum_Q\int_Q\tilde\rho_n(g-\bar g_Q)\right|
\le\|\tilde\rho_n\|_2\left(\sum_Q\|g-\bar g_Q\|_{L^2(Q)}^2\right)^{1/2}
\le C(2\cdot5^n)^{-1}\|\nabla g\|_2.
\]
Taking the supremum over $\|\nabla g\|_2\le1$ proves the $\dot H^{-1}$ bound.

\noindent\textbf{Step 4. Compact support.}
For each fixed $n$, the outermost tiles have no crossing on their outer side. The fixed straight endpoint neighbourhoods and the chosen tube width give a boundary strip of positive width where both fields vanish for all module times. Its width depends on $n$. Removing a smaller boundary strip leaves a compact set $K_n\subset(0,1)^2$ containing both supports. Thus extension by zero and periodic identification give smooth fields on $\mathbb T^2$.

\noindent\textbf{Step 5. Recursion relation.}
At $t=1$, the definition and the patching condition \eqref{1.2-4} give
\[
\tilde\rho_n(x,1)=\tilde\Theta_{i(Q')}\bigl(2\cdot5^{n+1}(x-r_{Q'}),0\bigr),
\qquad x\in Q',
\]
on each child $Q'$ of a parent tile $Q$.
Reindexing by $Q'\in\mathcal{Q}(2\cdot5^{n+1})$ (each contained in a unique $Q$) yields
\[
\tilde\rho_n(x,1)=\sum_{Q'}\chi_{Q'}(x)\tilde\Theta_{i(Q')}\bigl(2\cdot5^{n+1}(x-r_{Q'}),0\bigr)
=\tilde\rho_{n+1}(x,0).
\]
By smoothness, the identity also holds on tile boundaries, completing the proof of Theorem~\ref{1.2}.
\end{proof}

We now estimate the nonlinear term after perturbation.

\begin{theorem}\label{3.1}
For any $\alpha\in(0,1)$ and $k=0,1$, we have, almost everywhere,
\[
(\tilde{v}_n\cdot\nabla\tilde{v}_n)(x,t)=\sum_{Q\in\mathcal{Q}(2\cdot5^n)}\chi_Q(x)\,\frac{1}{2\cdot5^n}\,\tilde{V}_{i(Q)}\cdot\nabla\tilde{V}_{i(Q)}(2\cdot5^n(x-r_Q),t).
\]
Then there exists $C(\alpha,k)>0$ independent of $n,t,\epsilon$ such that
\begin{equation}
\|\partial_t^k(\tilde{v}_n\cdot\nabla\tilde{v}_n)\|_{C([0,1];C^\alpha(\mathbb{T}^2))}\le C(\alpha,k)5^{-(1-\alpha)n}.
\label{3-3}
\end{equation}
The corresponding difference from $v_n\cdot\nabla v_n$ is bounded by
$C(\alpha,k)\epsilon5^{-(1-\alpha)n}$ in the same norm.
\end{theorem}

\begin{proof}
By \eqref{3-2}, for $Q\in\mathcal Q(2\cdot5^n)$ set $z=2\cdot5^n(x-r_Q)$. Then
\[
(\tilde v_n\cdot\nabla\tilde v_n)(x,t)
=\frac{1}{2\cdot5^n}(\tilde V_{i(Q)}\cdot\nabla\tilde V_{i(Q)})(z,t).
\]
Let $\tilde W_i=\tilde V_i\cdot\nabla\tilde V_i$ and $W_i=V_i\cdot\nabla V_i$. Lemma~\ref{2.2} and the product rule give
\[
\|\partial_t^k\tilde W_i\|_{C_x^1}\le C,\qquad
\|\partial_t^k(\tilde W_i-W_i)\|_{C_x^1}\le C\epsilon,
\qquad k=0,1.
\]
For the difference, we differentiate
$\tilde W_i-W_i=(\tilde V_i-V_i)\cdot\nabla\tilde V_i+V_i\cdot\nabla(\tilde V_i-V_i)$. For $k=0,1$, rescaling gives an $L^\infty$ bound of $C(2\cdot5^n)^{-1}$ for $\partial_t^k(\tilde v_n\cdot\nabla\tilde v_n)$. Its spatial gradient is bounded by $C$. The corresponding difference estimates have an additional factor $\epsilon$. Applying \eqref{eq:holder-interpolation} gives \eqref{3-3} and the difference estimate, thus proving the theorem.
\end{proof}

Based on the perturbed basic family $(\tilde V_i,\tilde\Theta_i)_{i=1}^N$ and the quasi-self-similar patching of \cite{alberti2019exponential}, we have constructed the perturbed quasi-self-similar family $\{\tilde v_n,\tilde\rho_n\}$ and proved Theorem~\ref{1.2} together with the nonlinear estimate.

\section{Uniform boundedness of the forcing term}\label{sec:time}

In the next two sections, using the perturbed quasi-self-similar family from Section~3, we apply time smoothing to construct velocity fields $\tilde{v}^m$ and passive scalars $\tilde{\theta}^m$ satisfying the 2D forced Navier--Stokes equation and the corresponding transport equation with diffusion. The corresponding forcing terms $\tilde{g}^m$ are shown to be uniformly bounded in $C([0,1];C^\alpha)$ and converge to $\tilde{g}=\partial_t\tilde{v}+\tilde{v}\cdot\nabla\tilde{v}$. A frequency concentration lemma and a comparison argument up to a stopping time yield the uniform anomalous dissipation lower bound (Theorem~\ref{1.3}). Embedding this 2D result into the $(2+\frac12)$-dimensional framework proves Theorem~\ref{1.4}, namely, persistence of anomalous dissipation under the admissible perturbations in the forced-flow framework of Bru\`e and De Lellis \cite{brue2023anomalous}.

We now prove that the forcing term is uniformly bounded in $C([0,1];C^\alpha)$ for $\alpha\in(0,1)$.

\subsection{Time reparametrisation}
Following \cite{brue2023anomalous}, we use a time reparametrisation $\eta(t)$ to make the quasi-self-similar family smooth at the time junctions. Lemma~\ref{4.1} gives its properties.
\begin{lemma}\label{4.1}
Define $t_n=1-(n+1)^{-2}$ and $\Delta t_n=t_{n+1}-t_n$. There exists a non-decreasing function $\eta\in C^\infty([0,1))\cap C([0,1])$ such that
\begin{enumerate}
\item $\eta(t_n)=t_n$ for all $n\in\mathbb N$ and $\eta(1)=1$.
\item $\eta^{(k)}(t_n)=0$ for all $k\ge1$ and all $n\in\mathbb N$.
\item For every $k\ge1$ there exists $C(k)$ such that, for $t\in[t_n,t_{n+1}]$,
\[
|\eta^{(k)}(t)|\le C(k)(n+1)^{3(k-1)}.
\]
\end{enumerate}
\end{lemma}
\begin{proof}
Let $\omega(s)=\exp[-1/(s(1-s))]$ for $0<s<1$, extended by zero at the endpoints, and set
\[
\chi(s)=\frac{\int_0^s\omega(\xi)\,d\xi}{\int_0^1\omega(\xi)\,d\xi},\qquad
\eta(t)=t_n+\Delta t_n\chi\left(\frac{t-t_n}{\Delta t_n}\right),\quad t\in[t_n,t_{n+1}].
\]
Every derivative of $\chi$ of positive order vanishes at zero and one. The pieces therefore join smoothly at each $t_n$, and $\eta$ maps each time interval to itself. Set $\eta(1)=1$. Continuity follows from $t_n\uparrow1$. Since
\[
\begin{gathered}
\frac{3}{4(n+1)^3}\le\Delta t_n
=\frac{2n+3}{(n+1)^2(n+2)^2}\le\frac{2}{(n+1)^3},\\
\eta^{(k)}(t)=(\Delta t_n)^{1-k}\chi^{(k)}\left(\frac{t-t_n}{\Delta t_n}\right),
\end{gathered}
\]
the stated bounds follow.
\end{proof}

Using $\eta$, we smooth the quasi-self-similar family from Section~3. On each interval $[t_n,t_{n+1}]$, write
\[
\tau_n(t)=\frac{\eta(t)-t_n}{\Delta t_n},\qquad
A_n(t)=\frac{\eta'(t)}{\Delta t_n}.
\]
Here $\tau_n(t)\in[0,1]$. For $m\ge1$, define
\begin{equation}
 \tilde v^m(x,t)=
 \begin{cases}
 A_n(t)\tilde v_n(x,\tau_n(t)),
 &t\in[t_n,t_{n+1}),\quad 0\le n\le m,\\
 0,&t\in[t_{m+1},1],
 \end{cases}
 \label{4-2}
\end{equation}
\begin{equation}
 \tilde\rho^m(x,t)=
 \begin{cases}
 \tilde\rho_n(x,\tau_n(t)),
 &t\in[t_n,t_{n+1}),\quad 0\le n\le m,\\
 \tilde\rho_m(x,1),&t\in[t_{m+1},1].
 \end{cases}
 \label{4-3}
\end{equation}

Since all derivatives of $\eta$ of positive order vanish at the junctions $t_n$ and the scalar values agree by \eqref{1.3-5}, $\tilde v^m,\tilde\rho^m$ are smooth at these junctions. For each finite $m$, Theorem~\ref{1.2} therefore yields $\tilde v^m,\tilde\rho^m\in C^\infty([0,1];C^3(\mathbb T^2))$. By the chain rule, $\tilde\rho^m$ satisfies \eqref{4-18}, also on $[t_{m+1},1]$, where $\tilde v^m=0$ and $\tilde\rho^m$ is constant in time.

For $t<1$, the corresponding infinite-generation fields are
\begin{equation}
 \tilde\rho(x,t)=\tilde\rho_n(x,\tau_n(t)),
 \qquad t\in[t_n,t_{n+1}),\quad n\in\mathbb N,
\label{4-4}
\end{equation}
\begin{equation}
 \tilde v(x,t)=A_n(t)\tilde v_n(x,\tau_n(t)),
 \qquad t\in[t_n,t_{n+1}),\quad n\in\mathbb N.
\label{4-5}
\end{equation}
Set $\tilde v(1)=0$, defining an extension whose time regularity in $C^\alpha$ follows from the estimates below. The scalar $\tilde\rho$ has unit $L^2$ norm for $t<1$ and tends to zero in $\dot H^{-1}$ as $t\uparrow1$. It therefore converges weakly, but not strongly, to zero in $L^2$. In particular, smoothness on the closed time interval is used only for the finite truncations.

\subsection{Proof of uniform boundedness of the forcing term}

Choose the viscosity coefficient sequence
\begin{equation}
    \mu_m = m^{10}\,5^{-2m}.
    \label{4-6}
\end{equation}
Take $\tilde q^m=0$ and define the perturbed forcing
\begin{equation}
    \tilde{g}^m := \partial_t\tilde{v}^m + \tilde{v}^m\cdot\nabla\tilde{v}^m - \mu_m\Delta\tilde{v}^m,
    \label{4-7}
\end{equation}
which belongs to $C^\infty([0,1];C^1(\mathbb{T}^2))$ since $\tilde{v}^m\in C^\infty([0,1];C^3(\mathbb{T}^2))$. For each fixed $m$ and $h$, the fields are smooth. The estimates below use only the finite-order uniform bounds in Lemma~\ref{2.2}.

\begin{theorem}\label{4.2}
For any $\alpha\in(0,1)$ there exists a constant $\tilde C(\alpha)$ independent of $m$ and the admissible perturbation such that
\begin{equation}
    \|\tilde{g}^m\|_{C([0,1];C^\alpha(\mathbb{T}^2))}\le\tilde C(\alpha).
    \label{4-8}
\end{equation}
Moreover, as $m\to\infty$, $\tilde{g}^m$ converges in $C([0,1];C^\alpha(\mathbb{T}^2))$ to
\begin{equation}
    \tilde{g}:=\partial_t\tilde{v}+\tilde{v}\cdot\nabla\tilde{v}.
    \label{4-9}
\end{equation}
\end{theorem}

\begin{proof}
We decompose $\tilde{g}^m = \partial_t\tilde{v}^m + \tilde{v}^m\cdot\nabla\tilde{v}^m - \mu_m\Delta\tilde{v}^m$ and estimate each term separately. Each summand below is defined on $[t_n,t_{n+1}]$ and extended by zero elsewhere, and the vanishing of all positive-order derivatives of $\eta$ at the endpoints makes this construction smooth. Lemma~\ref{4.1} gives
\begin{equation}\label{eq:clock-bounds}
\|A_n\|_\infty\le C(n+1)^3,\qquad \|A_n'\|_\infty\le C(n+1)^6.
\end{equation}
At any time at most one interval contributes.

\emph{Time derivative.} The chain rule gives
\[
\partial_t\tilde v^m
=\sum_{n=0}^m\left(A_n'\tilde v_n(x,\tau_n)
 +A_n^2\partial_t\tilde v_n(x,\tau_n)\right).
\]
Applying Theorem~\ref{1.2} to $\tilde v_n$ and $\partial_t\tilde v_n$, we obtain
\[
\|\partial_t\tilde v^m\|_{C_tC_x^\alpha}
\le C(\alpha)\sup_{n\ge0}(n+1)^6 5^{-(1-\alpha)n}<\infty.
\]

\emph{Convective term.} Since $\eta(t)$ lies in a single interval $[t_l,t_{l+1})$ for $t<1$, only $n=l$ contributes when $l\le m$, giving
\[
\tilde v^m\cdot\nabla\tilde v^m
=\sum_{n=0}^m A_n^2(\tilde v_n\cdot\nabla\tilde v_n)(x,\tau_n).
\]
Theorem~\ref{3.1} and \eqref{eq:clock-bounds} give
\[
\|\tilde v^m\cdot\nabla\tilde v^m\|_{C_tC_x^\alpha}
\le C(\alpha)\sup_{n\ge0}(n+1)^6 5^{-(1-\alpha)n}<\infty.
\]

\emph{Viscous term.} By the Laplacian estimate in Theorem~\ref{1.2},
\begin{align*}
\|\mu_m\Delta\tilde v^m\|_{C_tC_x^\alpha}
&\le C(\alpha)\sup_{0\le n\le m}\mu_m(n+1)^3 5^{(1+\alpha)n}\\
&\le C(\alpha)(m+1)^{13}\,5^{-(1-\alpha)m}\longrightarrow0.
\end{align*}
Collecting the three estimates gives \eqref{4-8}, with a constant independent of $m$ and the admissible perturbation.

\emph{Convergence.} For $t\in[t_n,t_{n+1}]$, the infinite-generation field satisfies
\[
\|\tilde v(t)\|_{C^\alpha}\le C(\alpha)(n+1)^3 5^{-(1-\alpha)n},
\]
and
\[
\|\partial_t\tilde v(t)\|_{C^\alpha}
 +\|\tilde v(t)\cdot\nabla\tilde v(t)\|_{C^\alpha}
\le C(\alpha)(n+1)^6 5^{-(1-\alpha)n}.
\]
The fields in these estimates extend continuously to zero at $t=1$. Since $1-t\ge(n+2)^{-2}$, the first estimate gives $\|\tilde v(t)\|_{C^\alpha}/(1-t)\to0$ as $t\uparrow1$, so $\partial_t\tilde v(1)=0$. Thus $\tilde g=\partial_t\tilde v+\tilde v\cdot\nabla\tilde v$ belongs to $C([0,1];C^\alpha)$, with $\tilde g(1)=0$. For $t<t_{m+1}$, $\tilde g^m-\tilde g=-\mu_m\Delta\tilde v^m$, while for $t\ge t_{m+1}$, $\tilde g^m=0$. Hence
\[
\|\tilde g^m-\tilde g\|_{C_tC_x^\alpha}
\le C(\alpha)\left[\sup_{n>m}(n+1)^6 5^{-(1-\alpha)n}
 +(m+1)^{13}\,5^{-(1-\alpha)m}\right]\longrightarrow0.
\]
Applying the difference estimates in Theorems~\ref{1.2} and \ref{3.1} with the same clocks gives
\begin{equation}\label{eq:force-difference}
\sup_{m\ge1}\|\tilde g^m-g^m\|_{C_tC_x^\alpha}
 +\|\tilde g-g\|_{C_tC_x^\alpha}\le C(\alpha)\epsilon.
\end{equation}
Hence $\tilde g^m\to\partial_t\tilde v+\tilde v\cdot\nabla\tilde v=\tilde g$ in $C([0,1];C^\alpha)$, completing the proof.
\end{proof}
\section{Stability of the anomalous dissipation}\label{sec:dissipation}

In this section we prove Theorem \ref{1.3}. First we establish a frequency concentration estimate for $\tilde{\rho}$. We use Fourier coefficients $\hat f(k)=\int_{[0,1]^2}f(x)e^{-2\pi\mathrm i k\cdot x}\,dx$.

\begin{lemma}\label{4.3}
For $\tilde{\rho}$ defined in \eqref{4-4}, there exist $\Lambda\in(0,1)$ and $\tilde C>0$ depending only on the fixed data such that for all $n\in\mathbb{N}$, $t\in[t_n,t_{n+1})$,
\begin{equation}
    \|\nabla\tilde{\rho}(\cdot,t)\|_{L^2}\le \Lambda^{-1}5^n,
    \label{4-13}
\end{equation}
\begin{equation}
    \|\tilde{\rho}(\cdot,t)\|_{\dot H^{-1}}\le \tilde C\,5^{-n},
    \label{4-14}
\end{equation}
\begin{equation}
    \|P_{\le\Lambda5^n}\tilde{\rho}(\cdot,t)\|_{L^2}\le \frac14.
    \label{4-15}
\end{equation}
Here $P_{\le\Lambda5^n}$ is the Fourier projection onto frequencies $|2\pi k|\le\Lambda5^n$.
\end{lemma}

\begin{proof}
For $t\in[t_n,t_{n+1})$, set $\tau_n(t)=(\eta(t)-t_n)/\Delta t_n$. Using Theorem~\ref{1.2}, we obtain
\[
\|\nabla\tilde\rho(\cdot,t)\|_2
=\|\nabla\tilde\rho_n(\cdot,\tau_n(t))\|_2\le C_G5^n,
\qquad
\|\tilde\rho(\cdot,t)\|_{\dot H^{-1}}\le C_H5^{-n}.
\]
Set $\tilde C=C_H$. Then \eqref{4-14} holds, and any $\Lambda\le C_G^{-1}$ yields \eqref{4-13}.

For the low frequencies, using $\|\tilde{\rho}\|_{\dot H^{-1}}^2=\sum_{k\neq0}|2\pi k|^{-2}|\hat{\tilde{\rho}}(k)|^2$,
\[
\|P_{\le\Lambda5^n}\tilde{\rho}\|_{L^2}^2
\le (\Lambda5^n)^2\sum_{0<|2\pi k|\le\Lambda5^n}|2\pi k|^{-2}|\hat{\tilde{\rho}}(k)|^2
\le (\Lambda5^n)^2\|\tilde{\rho}\|_{\dot H^{-1}}^2.
\]
Substituting \eqref{4-14} yields
\[
\|P_{\le\Lambda5^n}\tilde{\rho}(\cdot,t)\|_{L^2} \le \Lambda5^n\|\tilde{\rho}(\cdot,t)\|_{\dot H^{-1}} \le  \Lambda\tilde C.
\] 
Choosing $\Lambda=\min\{1/2,C_G^{-1},(4\tilde C)^{-1}\}$ ensures the bound $1/4$, proving \eqref{4-15}.
\end{proof}

For the perturbed velocity field $\tilde{v}^m$, consider
\begin{equation}
\begin{cases}
\partial_t\tilde{\theta}^m+\tilde{v}^m\cdot\nabla\tilde{\theta}^m=\mu_m\Delta\tilde{\theta}^m,\\
\tilde{\theta}^m(x,0)=\tilde\rho(x,0),
\end{cases}
\label{4-17}
\end{equation}
Since $\tilde v^m$ and $\tilde\rho(\cdot,0)$ are smooth and $\mu_m>0$, classical parabolic theory gives a unique solution $\tilde\theta^m\in C^\infty(\mathbb T^2\times[0,1])$ for each fixed $m$ and admissible $h$.

We also consider the corresponding inviscid transport equation
\begin{equation}
\begin{cases}
\partial_t\tilde{\rho}^m+\tilde{v}^m\cdot\nabla\tilde{\rho}^m=0,\\
\tilde{\rho}^m(x,0)=\tilde\rho(x,0),
\end{cases}
\label{4-18}
\end{equation}
where $\tilde\rho^m$ is defined in \eqref{4-3}.

For $\tilde{v}^m$ and the solution $\tilde{\theta}^m$ of \eqref{4-17}, we need to prove
\begin{equation}
\liminf_{m\to\infty}\mu_m\int_0^1\|\nabla\tilde{\theta}^m(\cdot,t)\|_{L^2(\mathbb{T}^2)}^2\,dt \ge C_*>0,
\label{4-19}
\end{equation}
with $C_*$ independent of $m$ and of $h\in\mathcal H_{\epsilon_0}$. This estimate will be used in the $(2+\frac12)$-dimensional Navier--Stokes equations in Section~6.

\begin{lemma}\label{4.4}
Let $\tilde{\theta}^m$ be the solution of \eqref{4-17} and $\tilde{\rho}^m$ the solution of \eqref{4-18}, constant in time after $t_{m+1}$ as in \eqref{4-3}. Then for every $t\in[0,1]$,
\begin{equation}
\begin{aligned}
&\sup_{0\le s\le t}\|\tilde{\rho}^m(\cdot,s)-\tilde{\theta}^m(\cdot,s)\|_{L^2}^2\\
&\quad\le \left(2\mu_m\int_0^t\|\nabla\tilde{\theta}^m\|_{L^2}^2\,ds\right)^{1/2}
\left(2\mu_m\int_0^t\|\nabla\tilde{\rho}^m\|_{L^2}^2\,ds\right)^{1/2}.
\end{aligned}
\label{4-20}
\end{equation}
\end{lemma}

\begin{proof}
Define $\omega:=\tilde{\rho}^m-\tilde{\theta}^m$. Subtracting the two equations yields
\[
\partial_t\omega+\tilde{v}^m\cdot\nabla\omega=-\mu_m\Delta\tilde{\theta}^m,\qquad \omega(\cdot,0)=0.
\]
Forming the $L^2(\mathbb{T}^2)$ inner product with $\omega$ gives
\[
\frac12\frac{d}{dt}\|\omega\|_{L^2}^2+\int(\tilde{v}^m\cdot\nabla\omega)\omega\,dx
= -\mu_m\int(\Delta\tilde{\theta}^m)\omega\,dx.
\]
Since $\operatorname{div}\tilde{v}^m=0$, the convective term vanishes by integration by parts. Hence
\[
\frac12\frac{d}{dt}\|\omega\|_{L^2}^2 = -\mu_m\|\nabla\tilde{\theta}^m\|_{L^2}^2+\mu_m\int\nabla\tilde{\theta}^m\cdot\nabla\tilde{\rho}^m\,dx.
\]
Integrating in time and using $\omega(\cdot,0)=0$ together with the Cauchy--Schwarz inequality,
\[
\|\omega\|_{L^2}^2 \le \Bigl(2\mu_m\int_0^t\|\nabla\tilde{\theta}^m\|_{L^2}^2\,ds\Bigr)^{\!1/2}
\Bigl(2\mu_m\int_0^t\|\nabla\tilde{\rho}^m\|_{L^2}^2\,ds\Bigr)^{\!1/2}.
\]
Taking the supremum over $s\in[0,t]$ completes the proof.
\end{proof}

We now prove the stability of the anomalous dissipation under perturbation.

\begin{proof}[Proof of Theorem \ref{1.3}]
The proof proceeds by contradiction. By Theorem~\ref{1.2} and $\mu_m=m^{10}5^{-2m}$, we have
\begin{equation}\label{eq:stopping-time-existence}
\begin{aligned}
2\mu_m\int_0^{t_{m+1}}\|\nabla\tilde\rho^m(t)\|_2^2\,dt
&\ge 2c_G^2\mu_m5^{2m}\Delta t_m\\
&\ge\frac{3c_G^2m^{10}}{2(m+1)^3}\longrightarrow\infty.
\end{aligned}
\end{equation}
Choose $M\ge2$, independent of $h$, so that the left-hand side is greater than one for $m\ge M$. For each such $m$, continuity gives a first time $t^*\in(0,t_{m+1})$ such that
\begin{equation}\label{eq:stopping-time}
2\mu_m\int_0^{t^*}\|\nabla\tilde\rho^m(t)\|_2^2\,dt=1.
\end{equation}
Assume that
\[
\mu_m\int_0^1\|\nabla\tilde\theta^m(t)\|_2^2\,dt
<\min\left\{\frac1{512},\frac{\Lambda^4}{8}\right\}.
\]
Applying Lemma~\ref{4.4} on $[0,t^*]$ yields
\begin{equation}\label{eq:stopping-time-error}
\begin{aligned}
\sup_{0\le t\le t^*}\|\tilde\rho^m(t)-\tilde\theta^m(t)\|_2^2
&\le\left(2\mu_m\int_0^{t^*}\|\nabla\tilde\theta^m(t)\|_2^2\,dt\right)^{1/2}
\\
&<\frac1{16}.
\end{aligned}
\end{equation}
For $t\in[t_n,t_{n+1})\cap[0,t^*]$, we have $\tilde\rho^m(t)=\tilde\rho(t)$. Lemma~\ref{4.3}, the unit $L^2$ norm and \eqref{eq:stopping-time-error} give
\[
\begin{aligned}
\|P_{>\Lambda5^n}\tilde\theta^m(t)\|_2
&\ge\|P_{>\Lambda5^n}\tilde\rho^m(t)\|_2
-\|\tilde\rho^m(t)-\tilde\theta^m(t)\|_2\\
&\ge\frac{\sqrt{15}}4-\frac14\ge\frac12.
\end{aligned}
\]
Using the gradient upper bound in Lemma~\ref{4.3}, we obtain
\begin{equation}\label{eq:stopping-gradient-comparison}
\|\nabla\tilde\theta^m(t)\|_2^2
\ge\frac{\Lambda^2}{4}5^{2n}
\ge\frac{\Lambda^4}{4}\|\nabla\tilde\rho^m(t)\|_2^2.
\end{equation}
Integrating only over $[0,t^*]$ and using \eqref{eq:stopping-time}, we find
\begin{equation}\label{eq:stopping-dissipation}
\begin{aligned}
\mu_m\int_0^1\|\nabla\tilde\theta^m(t)\|_2^2\,dt
&\ge\mu_m\int_0^{t^*}\|\nabla\tilde\theta^m(t)\|_2^2\,dt\\
&\ge\frac{\Lambda^4}{4}\mu_m\int_0^{t^*}\|\nabla\tilde\rho^m(t)\|_2^2\,dt
=\frac{\Lambda^4}{8},
\end{aligned}
\end{equation}
which is a contradiction. Thus, with
\begin{equation}\label{eq:uniform-dissipation-constant}
C_*:=\min\left\{\frac1{512},\frac{\Lambda^4}{8}\right\}>0,
\end{equation}
the dissipation is at least $C_*$ for every $m\ge M$. Both $M$ and $C_*$ are independent of $h$. Therefore,
\[
\liminf_{m\to\infty}\mu_m\int_0^1\|\nabla\tilde\theta^m(t)\|_{L^2(\mathbb T^2)}^2\,dt
\ge C_*>0,
\]
which completes the proof of Theorem \ref{1.3}.
\end{proof}

\section{Stability of the anomalous dissipation for the forced 3D Navier--Stokes equations}

We now apply Theorem \ref{1.3} to the $(2+\frac12)$-dimensional framework to complete the proof of Theorem \ref{1.4}.

\begin{proof}[Proof of Theorem \ref{1.4}]
On $\mathbb{T}^3$ with coordinates $x=(x_1,x_2,x_3)$, define
\begin{align*}
\tilde{u}^{\mu_m}(x,t) &= \bigl(\tilde{v}^m(x_1,x_2,t),\; \tilde{\theta}^m(x_1,x_2,t)\bigr),\\
\tilde{f}^{\mu_m}(x,t) &= \bigl(\tilde{g}^m(x_1,x_2,t),\; 0\bigr),\\
\tilde{p}^{\mu_m}(x,t) &= \tilde{q}^m(x_1,x_2,t),
\end{align*}
where $\tilde{v}^m,\tilde{g}^m,\tilde{\theta}^m$ are the 2D fields and $\tilde{q}^m=0$ is the 2D pressure chosen in Section~4. This choice is consistent because the force is not required to be divergence-free. All quantities are independent of $x_3$. This construction reduces the 3D Navier--Stokes equations to the coupled 2D system
\begin{equation}
\begin{cases}
\partial_t \tilde v^{m} + \tilde v^{m} \cdot \nabla \tilde v^{m} + \nabla \tilde q^{m} = \mu_m \Delta \tilde {v}^m + \tilde g^{m}, \\
\operatorname{div} \tilde v^{m} = 0,\\
\partial_t \tilde \theta^{m} + \tilde v^{m} \cdot \nabla \tilde \theta^{m} = \mu_m \Delta \tilde \theta^m
\end{cases}.
\end{equation} Since the perturbation functions $h_i$ have zero initial values and the time change is flat at zero, the 3D initial datum $\tilde{u}^{\mu_m}(x,0)=(v_0,\theta_0)=(0,\rho_0(\cdot,0))\in C^{\infty}(\mathbb{T}^3)$ is independent of $m$ and the perturbation.

\noindent\textbf{Regularity.} By Theorem \ref{4.2}, $\tilde{g}^m$ is uniformly bounded in $C([0,1];C^\alpha(\mathbb{T}^2))$ and converges to $\tilde{g}$. Since $\tilde f^{\mu_m}=(\tilde g^m,0)$ is independent of $x_3$, we obtain
\begin{equation}
\|\tilde{f}^{\mu_m}\|_{C([0,1];C^\alpha(\mathbb{T}^3))}\le \|\tilde{g}^m\|_{C([0,1];C^\alpha(\mathbb{T}^2))}\le\tilde C(\alpha),
\end{equation}
with $\tilde C(\alpha)$ independent of $m$ and the admissible perturbation.

\noindent\textbf{Stability.} Define $\varepsilon_m:=\mu_m\int_0^1\int_{\mathbb{T}^3}|\nabla\tilde{u}^{\mu_m}|^2$. Independence of $x_3$ gives $|\nabla\tilde{u}^{\mu_m}|^2=|\nabla_{x_1,x_2}\tilde{v}^m|^2+|\nabla_{x_1,x_2}\tilde{\theta}^m|^2$, hence
\[
\varepsilon_m\ge\mu_m\int_0^1\int_{\mathbb{T}^2}|\nabla\tilde{\theta}^m|^2.
\]
Theorem \ref{1.3} provides $C_*>0$ with $\liminf_{m\to\infty}\mu_m\int_0^1\|\nabla\tilde\theta^m\|_{L^2(\mathbb T^2)}^2\,dt\ge C_*$. Therefore
\[
\liminf_{m\to\infty}\mu_m\int_0^1\|\nabla\tilde u^{\mu_m}\|_{L^2(\mathbb T^3)}^2\,dt\ge C_*.
\]
The lower bound for all sufficiently large $m$, with a threshold independent of $h$, follows from the proof of Theorem \ref{1.3}.
Together with the force difference estimate \eqref{eq:force-difference}, this completes the proof of Theorem \ref{1.4}.
\end{proof}

Thus, under pure normal perturbations satisfying conditions (i)--(v) with $\epsilon\le\epsilon_0$, the perturbed 3D Navier--Stokes solutions retain anomalous dissipation for the chosen viscosity sequence, with a lower bound uniform over the admissible perturbations.

\par\noindent
\textbf{Acknowledgments.} The author is deeply grateful to Prof. Liqun Zhang for his invaluable guidance and support throughout this research, which have been instrumental in shaping this work.

\par\noindent
\textbf{Author declarations.}

\par\noindent
\textbf{Conflict of interest.} The author has no conflicts to disclose.

\par\noindent
\textbf{Author contributions.} Changhong Li: Conceptualization, Formal analysis, Methodology, Writing -- original draft, and Writing -- review \& editing.

\par\noindent
\textbf{Declaration of Generative Artificial Intelligence (AI) Use.} During the preparation of the initial draft, the author used DeepSeek-V4 Preview (DeepSeek) in April–May 2026 for language polishing and translation. During revision, the author used GPT-6 Astra (OpenAI) in October–November 2026 for language polishing, translation, formatting, and checking references and mathematical notation. The author reviewed, verified, and approved the AI-assisted content and takes full responsibility for the entire manuscript, including all mathematical statements and proofs.

\par\noindent
\textbf{Data availability.} No new data were generated or analysed in support of this research. Data sharing is not applicable to this article.

\appendix
\section{Technical estimates}

\begin{lemma}[Lipschitz estimate for matrix inversion]\label{A.1}
Let $A,B\in\mathbb{R}^{2\times2}$ be invertible with a constant $K>0$ such that $\|A^{-1}\|\le K$ and $\|B^{-1}\|\le K$. Then
\begin{equation}
    \|A^{-1} - B^{-1}\| \le K^2 \|A - B\|.
\end{equation}
\end{lemma}
\begin{proof}
$A^{-1}-B^{-1}=A^{-1}(B-A)B^{-1}$, so taking norms gives
\[
\|A^{-1} - B^{-1}\| \le \|A^{-1}\|\cdot\|B - A\|\cdot\|B^{-1}\| \le K \cdot \|A - B\| \cdot K = K^2 \|A - B\|.
\]
\end{proof}

\begin{lemma}[Stability estimate for inverse derivative matrices]\label{A.2}
Let $\Phi, \tilde{\Phi}: \Omega \to \mathbb{R}^2$ be $C^2$ diffeomorphisms defined on an open set $\Omega\subset\mathbb{R}^2$ with $\det\nabla\Phi = \det\nabla\tilde{\Phi} = 1$, and assume there exists a constant $M>0$ such that
\[
\|\nabla\Phi\|_{C^0(\Omega)} \le M,\qquad \|(\nabla\Phi)^{-1}\|_{C^0(\Omega)} \le M.
\]
If $\|\nabla\tilde{\Phi} - \nabla\Phi\|_{C^0(\Omega)} \le \epsilon$ with $\epsilon$ sufficiently small, then there exists $C=C(M)$ such that
\begin{equation}
    \|(\nabla\tilde{\Phi})^{-1} - (\nabla\Phi)^{-1}\|_{C^0(\Omega)} \le C\epsilon.
\end{equation}
\end{lemma}
\begin{proof}
The Neumann series yields
\[
(\nabla\tilde{\Phi})^{-1}=(\nabla\Phi)^{-1}\sum_{k=0}^\infty\bigl[(\nabla\Phi-\nabla\tilde{\Phi})(\nabla\Phi)^{-1}\bigr]^k,
\]
hence
\[
(\nabla\tilde{\Phi})^{-1}-(\nabla\Phi)^{-1}=(\nabla\Phi)^{-1}\sum_{k=1}^\infty\bigl[(\nabla\Phi-\nabla\tilde{\Phi})(\nabla\Phi)^{-1}\bigr]^k.
\]
Taking norms and using $\|(\nabla\Phi)^{-1}\|\le M$,
\[
\|(\nabla\tilde{\Phi})^{-1}-(\nabla\Phi)^{-1}\|
\le M\cdot\frac{\|(\nabla\Phi-\nabla\tilde{\Phi})(\nabla\Phi)^{-1}\|}{1-\|(\nabla\Phi-\nabla\tilde{\Phi})(\nabla\Phi)^{-1}\|}.
\]
Since $\|(\nabla\Phi-\nabla\tilde{\Phi})(\nabla\Phi)^{-1}\|\le\|\nabla\Phi-\nabla\tilde{\Phi}\|\cdot\|(\nabla\Phi)^{-1}\|\le\epsilon M\le1/2$ ($\epsilon\le1/(2M)$),
\[
\|(\nabla\tilde{\Phi})^{-1}-(\nabla\Phi)^{-1}\|\le M\cdot\frac{\epsilon M}{1-1/2}=2M^2\epsilon.
\]
Setting $C = 2M^2$ completes the proof.
\end{proof}

\noindent\textit{Remark.} In Lemma~\ref{A.2}, we obtain an estimate for the inverse matrices at the same point of $\Omega$. The estimates for derivatives of the inverse maps are given in Lemma~\ref{lem:inverse-calculus}.\par

\begin{lemma}[Poincaré inequality]\label{A.3}
Let $Q\subset\mathbb{R}^2$ be an open square of side length $L$. For $g\in H^1(Q)$, define its average by $\bar{g}_Q = \frac{1}{|Q|} \int_Q g(x)\,dx$. Then there exists a constant $C_P>0$ independent of $L$ such that
\begin{equation}
    \|g - \bar{g}_Q\|_{L^2(Q)} \le C_P \cdot \operatorname{diam}(Q) \,\|\nabla g\|_{L^2(Q)},
\end{equation}
where $\operatorname{diam}(Q)=\sqrt{2}L$.
\end{lemma}

\begin{proof}
After a rigid rotation if necessary, scale $Q$ to the unit square $Q_0=(0,1)^2$ by $x=r_Q+Lz$ and define $\tilde{g}(z)=g(r_Q+Lz)$. Then $\bar{g}_Q=\bar{\tilde{g}}_{Q_0}$ and
\[
\|g-\bar{g}_Q\|_{L^2(Q)}^2=L^2\|\tilde{g}-\bar{\tilde{g}}_{Q_0}\|_{L^2(Q_0)}^2,\quad
\|\nabla g\|_{L^2(Q)}^2=\|\nabla_z\tilde{g}\|_{L^2(Q_0)}^2.
\]
The classical Poincaré inequality on $Q_0$ gives
\[
\|\tilde{g} - \bar{\tilde{g}}_{Q_0}\|_{L^2(Q_0)} \le C \|\nabla_z \tilde{g}\|_{L^2(Q_0)},
\]
with $C$ independent of $L$. Consequently,
\[
\|g - \bar{g}_Q\|_{L^2(Q)} \le L C \|\nabla g\|_{L^2(Q)}.
\]
Substituting $L = \operatorname{diam}(Q)/\sqrt{2}$ yields the desired inequality by taking $C_P = C/\sqrt{2}$.
\end{proof}

\begin{lemma}[Finite-order mixed calculus]\label{lem:calculus}
Let $1\le p\le k$. Products and spatial compositions satisfy the following estimates on compact coordinate regions. Using the notation of Section~\ref{sec:complete-fields},
\[
\N_{p,k}(gP)\le C\N_{p,k}(g)\N_{p,k}(P),\qquad
\N_{p,k}((gP)\circ\Psi)\le C\N_{p,k}(gP),
\]
where $\N_{p,k}(\Psi)$ is uniformly bounded and the range of $\Psi$ lies in a compact subset of the parameter domain. If $P,\Psi$ are fixed and smooth, $\tilde P,\tilde\Psi$ have uniform $\N_{p,k}$ bounds, and $\N_{p,k}(\tilde\Psi-\Psi)$ is sufficiently small, then
\[
\N_{p,k}\bigl((g\tilde P)\circ\tilde\Psi-(gP)\circ\Psi\bigr)
\le C\bigl(\N_{p,k}(\tilde P-P)+\N_{p,k}(\tilde\Psi-\Psi)\bigr).
\]
The same estimates hold for other products and compositions, with the time variable unchanged. For the difference estimate, $C$ may depend on one additional spatial derivative of the fixed outer function.
\end{lemma}
\begin{proof}
The Leibniz rule gives the product estimate. For compositions, the chain rule gives
\begin{align*}
D_x((gP)\circ\Psi)&=(D_z(gP)\circ\Psi)D_x\Psi,\\
\partial_t((gP)\circ\Psi)&=(\partial_t(gP))\circ\Psi
 +(D_z(gP)\circ\Psi)\partial_t\Psi.
\end{align*}
Applying $\partial_t^aD_x^\beta$ to $(gP)\circ\Psi$, with $a\le p$ and $a+|\beta|\le k$, gives products whose factors have time order at most $a$ and total order at most $a+|\beta|$. This gives the composition estimate.

Using the mean-value formula, we obtain
\begin{align*}
&(g\tilde P)\circ\tilde\Psi-(gP)\circ\Psi\\
&\quad=[g(\tilde P-P)]\circ\tilde\Psi
 +\int_0^1D_z(gP)\bigl(t,\Psi+\lambda(\tilde\Psi-\Psi)\bigr)
                  (\tilde\Psi-\Psi)\,\mathrm d\lambda.
\end{align*}
The estimates above give the bound, with one extra spatial derivative only for the fixed $gP$. The interior margin and the smallness of $\tilde\Psi-\Psi$ allow the calculation on finitely many common convex coordinate neighbourhoods. Subtracting one factor at a time proves the product difference estimate.
\end{proof}

\begin{lemma}[Stability estimates for inverse maps]\label{lem:inverse-calculus}
Let $1\le p\le k$, $k\ge2$, and let $\tilde\Phi(t,\cdot)$ be diffeomorphisms on a common parameter domain, sufficiently close in $\N_{p,k}$ to a fixed smooth diffeomorphism $\Phi(t,\cdot)$. Denote their inverses by $\tilde\Psi$ and $\Psi$. On a common physical region, assume that both inverse maps take values in a fixed compact subset of the parameter domain and that these values are covered by finitely many common uniform inverse-function charts, with
$\|(D_z\Phi)^{-1}\|_{C^0}+\|(D_z\tilde\Phi)^{-1}\|_{C^0}\le C$.
Then
\[
\N_{p,k}(\tilde\Psi)\le C,\qquad
\N_{p,k}(\tilde\Psi-\Psi)\le C\N_{p,k}(\tilde\Phi-\Phi).
\]
The constants depend on the fixed reference and the stated bounds, without requiring higher derivatives of $\tilde\Phi$.
\end{lemma}
\begin{proof}
On sufficiently small parameter balls, uniform $C^1$ closeness gives
$|\tilde\Phi(t,z)-\tilde\Phi(t,z')|\ge c|z-z'|$.
Both preimages lie in the same such ball, so
\[
\|\tilde\Psi-\Psi\|_{C^0}
\le C\|(\Phi-\tilde\Phi)\circ\Psi\|_{C^0}
\le C\|\tilde\Phi-\Phi\|_{C^0}.
\]
The chain rule, Lemma~\ref{A.1} and the $C^0$ estimate give
\begin{gather*}
D_x\tilde\Psi=(D_z\tilde\Phi\circ\tilde\Psi)^{-1},\qquad
\partial_t\tilde\Psi=-(D_z\tilde\Phi\circ\tilde\Psi)^{-1}
                     (\partial_t\tilde\Phi\circ\tilde\Psi),\\
\|D_x\tilde\Psi-D_x\Psi\|_{C^0}
 +\|\partial_t\tilde\Psi-\partial_t\Psi\|_{C^0}
\le C\N_{1,1}(\tilde\Phi-\Phi).
\end{gather*}

For $2\le a+|\beta|\le k$ and $a\le p$, differentiating $\tilde\Phi(t,\tilde\Psi(t,x))=x$ gives the highest-order term $(D_z\tilde\Phi\circ\tilde\Psi)\partial_t^aD_x^\beta\tilde\Psi$.
The remaining factors have total orders at most $a+|\beta|-1$ for $\tilde\Psi$ and $a+|\beta|$ for $\tilde\Phi$, and time orders at most $a$. The bounded inverse differential then gives the uniform estimates by induction.

For the differences, subtract the corresponding identity for $\Phi,\Psi$. When $r\le a$ and $r+|\gamma|\le a+|\beta|$, we have
\begin{align*}
&\|(\partial_t^rD_z^\gamma\tilde\Phi)\circ\tilde\Psi
       -(\partial_t^rD_z^\gamma\Phi)\circ\Psi\|_{C^0}\\
&\quad\le\|\partial_t^rD_z^\gamma(\tilde\Phi-\Phi)\|_{C^0}
 +C\|\tilde\Psi-\Psi\|_{C^0}
\le C\N_{p,k}(\tilde\Phi-\Phi),
\end{align*}
where only the fixed $\Phi$ requires one extra spatial derivative. Lemma~\ref{A.1}, induction and the finite covering give the stated estimate.
\end{proof}

\bibliographystyle{abbrv}
\bibliography{anomalous_dissipation_revised_references}

@article{alberti2014exponential,
  title={Exponential self-similar mixing and loss of regularity for continuity equations},
  author={Alberti, Giovanni and Crippa, Gianluca and Mazzucato, Anna L},
  journal={Comptes Rendus. Math{\'e}matique},
  volume={352},
  number={11},
  pages={901--906},
  year={2014}
}

@article{alberti2019exponential,
  title={Exponential self-similar mixing by incompressible flows},
  author={Alberti, Giovanni and Crippa, Gianluca and Mazzucato, Anna L.},
  journal={Journal of the American Mathematical Society},
  volume={32},
  number={2},
  pages={445--490},
  year={2019}
}

@article{alberti2019loss,
  title={Loss of regularity for the continuity equation with non-{Lipschitz} velocity field},
  author={Alberti, Giovanni and Crippa, Gianluca and Mazzucato, Anna L},
  journal={Annals of PDE},
  volume={5},
  number={1},
  note={Article 9},
  year={2019},
  publisher={Springer}
}

@article{armstrong2023anomalous,
  title={Anomalous diffusion by fractal homogenization},
  author={Armstrong, Scott and Vicol, Vlad},
  journal={Annals of PDE},
  volume={11},
  note={Article 2},
  year={2025},
  doi={10.1007/s40818-024-00189-6}
}

@article{brue2022onsager,
  title={{Onsager} critical solutions of the forced {Navier--Stokes} equations},
  author={Bru{\`e}, Elia and Colombo, Maria and Crippa, Gianluca and De Lellis, Camillo and Sorella, Massimo},
  journal={Communications on Pure and Applied Analysis},
  volume={23},
  number={10},
  pages={1350--1366},
  year={2024},
  doi={10.3934/cpaa.2023071}
}

@article{brue2023anomalous,
  title={Anomalous dissipation for the forced {3D} {Navier--Stokes} equations},
  author={Bru{\`e}, Elia and De Lellis, Camillo},
  journal={Communications in Mathematical Physics},
  volume={400},
  number={3},
  pages={1507--1533},
  year={2023},
  publisher={Springer}
}

@article{colombo2022anomalous,
  title={Anomalous dissipation and lack of selection in the {Obukhov--Corrsin} theory of scalar turbulence},
  author={Colombo, Maria and Crippa, Gianluca and Sorella, Massimo},
  journal={Annals of PDE},
  volume={9},
  note={Article 21},
  year={2023},
  doi={10.1007/s40818-023-00162-9}
}

@article{drivas2022anomalous,
  title={Anomalous dissipation in passive scalar transport},
  author={Drivas, Theodore D and Elgindi, Tarek M and Iyer, Gautam and Jeong, In-Jee},
  journal={Archive for Rational Mechanics and Analysis},
  volume={243},
  number={3},
  pages={1151--1180},
  year={2022},
  publisher={Springer}
}

@article{jeong2021vortex,
  title={Vortex stretching and enhanced dissipation for the incompressible {3D} {Navier--Stokes} equations},
  author={Jeong, In-Jee and Yoneda, Tsuyoshi},
  journal={Mathematische Annalen},
  volume={380},
  number={3--4},
  pages={2041--2072},
  year={2021},
  publisher={Springer}
}

@article{jeong2022quasi,
  title={Quasi-streamwise vortices and enhanced dissipation for incompressible {3D} {Navier--Stokes} equations},
  author={Jeong, In-Jee and Yoneda, Tsuyoshi},
  journal={Proceedings of the American Mathematical Society},
  volume={150},
  number={3},
  pages={1279--1286},
  year={2022}
}

@article{kaneda2003energy,
  title={Energy dissipation rate and energy spectrum in high resolution direct numerical simulations of turbulence in a periodic box},
  author={Kaneda, Yukio and Ishihara, Takashi and Yokokawa, Mitsuo and Itakura, Ken’ichi and Uno, Atsuya},
  journal={Physics of Fluids},
  volume={15},
  number={2},
  pages={L21--L24},
  year={2003},
  publisher={American Institute of Physics}
}

@article{kolmogorov1941dissipation,
  title={Dissipation of energy in the locally isotropic turbulence},
  author={Kolmogorov, Andrey Nikolaevich},
  journal={C. R. Acad. Sci. URSS},
  volume={32},
  pages={16--18},
  year={1941}
}

@article{kolmogorov1941local,
  title={The local structure of turbulence in incompressible viscous fluid for very large {Reynolds} numbers},
  author={Kolmogorov, Andrey Nikolaevich},
  journal={C. R. Acad. Sci. URSS},
  volume={30},
  pages={301--305},
  year={1941}
}

@article{kolmogorov1941degeneration,
  title={On degeneration (decay) of isotropic turbulence in an incompressible viscous liquid},
  author={Kolmogorov, Andrej Nikolaevich},
  journal={C. R. Acad. Sci. URSS},
  volume={31},
  pages={538--540},
  year={1941}
}

@book{richardson1922weather,
  title={Weather prediction by numerical process},
  author={Richardson, Lewis F},
  year={1922},
  publisher={Cambridge University Press}
}

@article{sreenivasan1998update,
  title={An update on the energy dissipation rate in isotropic turbulence},
  author={Sreenivasan, Katepalli R},
  journal={Physics of Fluids},
  volume={10},
  number={2},
  pages={528--529},
  year={1998},
  publisher={American Institute of Physics}
}

@book{zeidler1993nonlinear,
  title={Nonlinear Functional Analysis and its Applications. I: Fixed-Point Theorems},
  author={Zeidler, Eberhard},
  year={1986},
  publisher={Springer-Verlag},
  address={New York},
  note={Translated from the German by Peter R. Wadsack}
}
\end{document}